\documentclass[options]{asl}

\usepackage{stmaryrd}
\usepackage{hyperref}

\title[Deep $\Pi^0_1$ classes]
{Deep $\Pi^0_1$ classes}


\author{Laurent Bienvenu}
\revauthor{Bienvenu, Laurent}

\address{Laboratoire CNRS J.-V. Poncelet, 119002, Bolshoy Vlasyevskiy Pereulok 11, Moscow, Russia}

\email{laurent.bienvenu@computability.fr}

\author{Christopher P.\ Porter}
\revauthor{Porter, Christopher P.}

\address{Department of Mathematics, University of Florida,
Gainesville, Florida 32611- 8105, USA
}

\email{cp@cpporter.com} 

\thanks{Both authors acknowledge the support of the John Templeton Foundation through the grant ``Structure and Randomness in the Theory of Computation."}

\newpage

\newtheorem{theorem}{Theorem}[section]

\newtheorem{proposition}[theorem]{Proposition}
\newtheorem{corollary}[theorem]{Corollary}

\newtheorem{lemma}[theorem]{Lemma}

\theoremstyle{definition}
\newtheorem{definition}[theorem]{Definition}

\newtheorem{fact}[theorem]{Fact}
\newtheorem{remark}[theorem]{Remark}
\newtheorem{question}{Question}

\newcommand{\C}{\mathcal{C}}
\newcommand{\CF}{\mathcal{CF}}
\newcommand{\D}{\mathcal{D}}
\newcommand{\E}{\mathcal{E}}

\newcommand{\F}{\mathcal{F}}

\newcommand{\I}{\mathbf{I}}
\newcommand{\K}{\mathrm{K}}
\newcommand{\KK}{\mathcal{K}}

\renewcommand{\L}{\mathcal{L}}
\newcommand{\M}{\mathbf{M}}
\newcommand{\Mb}{\overline{\mathbf{M}}}
\newcommand{\m}{\mathbf{m}}
\newcommand{\N}{\mathbb{N}}
\renewcommand{\O}{\mathcal{O}}
\renewcommand{\P}{\mathcal{P}}
\newcommand{\Q}{\mathbb{Q}}

\newcommand{\sS}{\mathcal{S}}

\newcommand{\U}{\mathcal{U}}

\newcommand{\V}{\mathcal{V}}
\newcommand{\W}{\mathcal{W}}

\newcommand{\PA}{\mathsf{PA}}
\newcommand{\sPA}{\mathcal{PA}}

\newcommand{\DNC}{\mathcal{DNC}}

\newcommand{\uh}{{\upharpoonright}}

\newcommand{\KR}{\mathsf{KR}}
\newcommand{\MLR}{\mathsf{MLR}}
\newcommand{\WtwoR}{\mathsf{W2R}}
\newcommand{\DiffR}{\mathsf{DiffR}}

\newcommand{\CR}{\mathsf{CR}}

\newcommand{\Low}{\mathrm{Low}}
\newcommand{\High}{\mathrm{High}}

\newcommand{\cs}{2^\N}
\newcommand{\str}{2^{<\N}}

\newcommand{\emptystring}{\Lambda}
\renewcommand{\tt}{\mathit{tt}}

 \newcommand{\la}{\langle}
\newcommand{\ra}{\rangle}
 \newcommand{\llb}{\llbracket}
\newcommand{\rrb}{\rrbracket}

\newcommand{\binrat}{\mathbb{Q}_2}

\newcommand{\dom}{\mathsf{dom}}
\newcommand{\atom}{\mathsf{Atom}}

\begin{document}

\begin{abstract}
A set of infinite binary sequences $\C\subseteq\cs$ is \emph{negligible} if there is no partial probabilistic algorithm that produces an element of this set with positive probability. The study of negligibility is of particular interest in the context of $\Pi^0_1$ classes. In this paper, we introduce the notion of depth for $\Pi^0_1$ classes, which is a stronger form of negligibility. Whereas a negligible $\Pi^0_1$ class $\C$ has the property that one cannot probabilistically compute a member of $\C$ with positive probability, a deep $\Pi^0_1$ class $\C$ has the property that one cannot probabilistically compute an \emph{initial segment} of a member of $\C$ with \emph{high} probability.  That is, the probability of computing a length $n$ initial segment of a deep $\Pi^0_1$ class converges to~0 effectively in $n$.

We prove a number of basic results about depth, negligibility, and a variant of negligibility that we call $\tt$-negligibility.  We provide a number of examples of deep $\Pi^0_1$ classes that occur naturally in computability theory and algorithmic randomness.  We also study deep classes in the context of mass problems, examine the relationship between deep classes and certain lowness notions in algorithmic randomness, and establish a relationship between members of deep classes and the amount of mutual information with Chaitin's $\Omega$. 
\end{abstract}

\maketitle

\tableofcontents

\section{Introduction} 

Much of the work carried out in computability theory since its inception has been concerned with deterministic computation, that is, computational procedures with the property that after any finite number of steps have been carried out, there is one unique step that follows from them. In contrast, the idea of probabilistic computation, which is prominent in computational complexity, rarely comes up in computability theory. One of the reason for this is the early result of De Leeuw, Moore, Shannon and Shapiro~\cite{DeLeeuwMSS1956} (independently proven by Sacks~\cite{Sacks1963}), which states that if an infinite sequence can be generated with positive probability by a probabilistic computation, then it can in fact be generated deterministically. However, there are some \emph{sets}~$\C$ of infinite sequences which contain no computable element, but such that a probabilistic computation can produce a member of $\C$ with positive probability. A trivial example is the set $\C$ of non-computable sequences: it suffices to pick each bit at random (with equal probability for $0$ and $1$, independently of the other bits), and since there are only countably many computable sequences, the resulting sequence will be non-computable with probability one.


The basic objects of the present investigation are $\Pi^0_1$ classes, that is, effectively closed subclasses of $\cs$.  Our goal is to study the properties of those $\Pi^0_1$ classes whose members cannot be probabilistically obtained with positive probability.  For our purposes, the model of probabilistic computation that is most convenient is given by oracle Turing machines equipped with algorithmically random oracles.  Thus, for a given $\Pi^0_1$ class $\C\subseteq\cs$, the question
\begin{itemize}
\item[] Can one probabilistically compute some member of $\C$ with positive probability?
\end{itemize}
amounts to asking
\begin{itemize}
\item[] Does the class of oracles that compute some member of $\C$ have positive Lebesgue measure?
\end{itemize}

We will discuss two specific types of $\Pi^0_1$ classes, \emph{negligible classes} and \emph{deep classes}, focussing in particular on the latter type, which is a special case of the former.

Roughly, a $\Pi^0_1$ class~$\C$ is \emph{negligible} if the probability of producing a member of~$\C$ via any probabilistic algorithm is 0.  Rephrased in terms of the model of probabilistic computation we use here, $\C$ is negligible if the collection
$\bigcup_{e\in\N}\Phi_e^{-1}(\C)$ has Lebesgue measure~0, where $(\Phi_e)_{e\in\N}$ is an effective enumeration of all Turing functionals.  The first example of a negligible class appeared in \cite{JockuschS1972}, in which it was proved that the $\Pi^0_1$ class of consistent completions of Peano arithmetic is negligible.  Negligibility was further studied in  \cite{Levin1984}, \cite{LevinV1977}, and \cite{Vyugin1982} (the term \emph{negligible} is explicitly introduced in the latter reference).

Depth can be viewed as a strong form of negligibility.  A deep $\Pi^0_1$ class $\C$ has the property that producing an initial segment of some member of $\C$ is in some sense maximally difficult: the probability of obtaining such an initial segment not only converges to 0 as we consider longer initial segments of members of $\C$, but this convergence is fast, as we can effectively bound the rate of convergence.  That initial segments of members of deep classes are so difficult to produce via probabilistic computation reflects the fact that members of deep classes are highly structured: initial segments of these sequences cannot be successfully produced by any combination of Turing machine and random oracle.  


Although the notion of depth is isolated for the first time in the present study, it is implicitly used in the work of both Levin \cite{Levin2013} and Stephan \cite{Stephan2002}.  In fact, in each of these papers one can extract a proof that the consistent completions of Peano arithmetic form a deep class, a result that we reprove in Section \ref{sec:examples}.  

The primary goals of this paper are (1) to prove a number of basic results about deep $\Pi^0_1$ classes, (2) to determine the exact relationship between negligibility and depth (and a related notion we call \emph{$\tt$-negligibility}), and (3) to provide a number of examples of deep $\Pi^0_1$ classes that occur naturally in computability theory and algorithmic randomness.  

The outline of this paper is as follows:  In Section \ref{sec:background}, we provide the technical background for the remainder of the study.  The notions of negligibility and $\tt$-negligibili\-ty are introduced and compared in Section \ref{sec:neg}, while the notions of depth and $\tt$-depth are introduced in Section \ref{sec:depth}. In~Section \ref{sec:depth}, we also separate depth from negligibility, show the equivalence between $\tt$-depth and $\tt$-negligibility, and prove several basic facts about these various kinds of classes.  Section \ref{sec:comp-limits} contains a brief discussion of the levels of randomness that guarantee that a random sequence cannot compute any member of any $\tt$-negligible, negligible, and deep class.  We consider the notions of depth and negligibility in the context of Medvedev and Muchnik reducibility in Section \ref{sec:mass-problems}.  In Section \ref{sec:examples} we present six examples  of families of deep classes, each of which is defined in terms of some well-studied notion from computability theory and algorithmic randomness:  completions of arithmetic, shift-complex sequences, diagonally non-computable functions, compression functions, finite sets of maximally incompressible strings, and dominating martingales related to the class $\High(\CR,\MLR)$.  In Section \ref{sec:lowness}, we establish connections between depth, $\tt$-depth and several lowness notions.  Lastly, in Section \ref{sec:IP}, we apply the notion of mutual information to deep classes, generalizing a result of Levin's from~\cite{Levin2013}.

\section{Background}\label{sec:background}

\subsection{Notation}\label{subsec:notation}

Let us fix some notation and terminology.  We denote by $\cs$ the set of infinite binary sequences (which we often refer to as ``sequences"), also known as Cantor space. We denote the set of finite strings by~$\str$ and the empty string by~$\emptystring$. $\binrat$ is the set of dyadic rationals, i.e., multiples of a negative power of~$2$.  Given $X\in \cs$ and an integer~$n$, $X \uh n$ is the string that consists of the first~$n$ bits of~$X$, and $X(n)$ is the $(n+1)^\mathrm{st}$ bit of~$X$ (so that $X(0)$ is the first bit of $X$).  For integers $n$ and $k$, $X \uh [n,n+k]$ denotes the subword $\sigma$ of $X$ of length $k+1$ such that for $i \leq k$, $\sigma(i)=X(i+n)$.

If $\sigma$ is a string and $x\in\str\cup\cs$, then $\sigma \preceq x$ means that $\sigma$ is a prefix of~$x$. A prefix-free set of strings is a set of strings such that none of its elements is a strict prefix of another one. For strings $\sigma,\tau\in\str$, $\sigma^\frown\tau$ denotes the string obtained by concatenating $\sigma$ and $\tau$.  Similarly, for $\sigma\in\str$ and $X\in\cs$, $\sigma^\frown X$ is the sequence obtained by concatenating $\sigma$ and $X$. For $X,Y\in\cs$, $X\oplus Y\in\cs$ satisfies $X\oplus Y(2n)=X(n)$ and $X\oplus Y(2n+1)=Y(n)$.

Given a string~$\sigma$, the \emph{cylinder} $\llb\sigma\rrb$ is the set of elements of $\cs$ having~$\sigma$ as a prefix.  Moreover, given $S\subseteq\str$, $\llb S\rrb$ is defined to be the set $\bigcup_{\sigma\in S}\llb\sigma\rrb$. When we refer to the topology of the Cantor space, we implicitly mean the product topology, i.e., the topology whose open sets are exactly those of type $\llb S\rrb$ for some~$S$. For this topology, some sets are both open and closed (clopen): these are the sets of type $\llb S\rrb$ when~$S$ is finite. 

A \emph{tree} is a set of strings that is closed downwards under the prefix relation. A \emph{path} through a tree~$T$ is a member of~$\cs$ all of whose prefixes are in~$T$. The set of paths of a tree~$T$ is denoted by~$[T]$. The \emph{$n^\mathrm{th}$ level of a tree}~$T$, denoted $T_n$, is the set of members of~$T$ of length~$n$.  An effectively open set, also called a $\Sigma^0_1$ class, is a set of type $\llb S \rrb$ for some c.e.\ set of strings~$S$. An effectively closed set, or $\Pi^0_1$ class, is the complement of some effectively open set. It is well known that a class is $\Pi^0_1$ if and only if it is of the form~$[T]$ for some co-c.e. (or computable) tree~$T$. Given a $\Pi^0_1$ class~$\C$, its \emph{canonical co-c.e.\ tree} is the tree $T=\{\sigma \colon \llb\sigma\rrb \cap \C \not= \emptyset\}$. The arithmetic hierarchy is defined inductively: a set is $\Sigma^0_{n+1}$ if it is a uniform union of $\Pi^0_n$ sets, and a set is $\Pi^0_{n+1}$ if it is a uniform intersection of~$\Sigma^0_n$ sets. 

A real number $r\in[0,1]$ is \emph{left-c.e.}\ (resp.\ \emph{right-c.e.}) if it is the limit of a non-decreasing (resp.\ non-increasing) computable sequence of rationals.

An \emph{order function} is a function $h: \N \rightarrow \N$ that is non-decreasing and unbounded. Given an order function~$h$, the \emph{inverse of~$h$}, denoted $h^{-1}$ is the order function defined as follows: For all~$k$, $h^{-1}(k)$ is the smallest $n$ such that $h(n) \geq k$. Note that $h^{-1}$ is computable if~$h$ is.

We adopt standard computability notation: $\leq_T$ denotes Turing reducibility, $\leq_{tt}$ denotes $\tt$-reducibility, $A'$ is the Turing jump of~$A$ (and $\emptyset'$ denotes the jump of the zero sequence). Throughout the paper, when an object~$A$ has a computable enumeration or limit approximation, we use the notation~$A[s]$ to denote the approximation of the object at stage~$s$. Moreover, if $A$ contains several expressions that have a computable approximation, the notation~$A[s]$ means that all of these expressions are approximated up to stage~$s$. For example, if $T$ is a tree and~$f$ is a function, both of which have computable approximation, $T_{f(n)}[s]$ is equal to $(T[s])_{f(n)[s]}$.

All logarithms will be taken with respect to the base 2.  Finally, we adopt the following asymptotic notation. For two functions $f,g: \N \rightarrow \N$, we sometimes write $f \leq^+ g$ to abbreviate $f \leq g + O(1)$ and $f \leq^\times g$ to abbreviate $f = O(g)$.

\subsection{Measures on $\cs$}\label{subsec:measures}

Recall that by Caratheodory's Theorem, a measure $\mu$ on $\cs$ is uniquely determined by specifying the values of $\mu$ on the basic open sets of $\cs$, where $\mu(\llb\sigma\rrb)=\mu(\llb\sigma0\rrb)+\mu(\llb\sigma1\rrb)$ for every $\sigma\in\str$.  If we further require that $\mu(\cs)=1$, then $\mu$ is a \emph{probability measure}. Hereafter, we will write $\mu(\llb\sigma\rrb)$ as $\mu(\sigma)$.  In addition, for a set $S\subseteq\str$, we will write $\mu(S)$ as shorthand for $\mu(\llb S\rrb)$.  In the case that $S$ is prefix-free, we will have $\mu(S)=\sum_{\sigma\in\str} \mu(\sigma)$.

The \emph{uniform (or Lebesgue) measure}~$\lambda$ is the unique Borel measure such that $\lambda(\sigma)=2^{-|\sigma|}$ for all strings~$\sigma$.  A measure $\mu$ on $\cs$ is \emph{computable} if $\sigma \mapsto \mu(\sigma)$ is computable as a real-valued function, i.e., if there is a computable function $\tilde \mu:\str\times\N\rightarrow\binrat$ such that
\[|\mu(\sigma)-\tilde \mu(\sigma,i)|\leq 2^{-i}\]
for every $\sigma\in\str$ and $i\in\N$.    

One family of examples of computable measures is given by the collection of Dirac measures concentrated on some computable point.  That is, if $X\in\cs$ is a computable sequence, then the \emph{Dirac measure concentrated on $X$}, denoted $\delta_X$, is defined as follows:
\[
\delta_X(\sigma)=
\left\{
	\begin{array}{ll}
	1  & \mbox{if } \sigma\prec X\\
	0 & \mbox{if } \sigma\not\prec X
	\end{array}.
\right.
\]
More generally, for a measure $\mu$, we say that $X\in\cs$ is an \emph{atom of $\mu$} or a \emph{$\mu$-atom}, denoted $X\in\atom_\mu$, if $\mu(\{X\})>0$.  Kautz proved the following:

\begin{lemma}[Kautz \cite{Kautz1991}]\label{lem-kautz}
$X\in\cs$ is computable if and only if $X$ is an atom of some computable measure.
\end{lemma}

There is a close connection between computable measures and 
 a certain class of Turing functionals.  Recall that a \emph{Turing functional} $\Phi:\subseteq\cs\rightarrow\cs$ may be defined from a c.e.\ set $S_\Phi$ of pairs of strings $(\sigma,\tau)$ such that 
if $(\sigma,\tau),(\sigma',\tau')\in S_\Phi$ and $\sigma\preceq\sigma'$, then $\tau\preceq\tau'$ or $\tau'\preceq\tau$.  For a Turing functional $\Phi$, we will always include $(\emptystring,\emptystring)\in S_\Phi$.

For each $\sigma\in\str$, we define $\Phi^\sigma$ to be the maximal string  in $\{\tau: (\exists \sigma'\preceq\sigma)((\sigma',\tau)\in S_\Phi)\}$ in the order given by $\preceq$.  To obtain a map defined on $\cs$ from the c.e.\ set of pairs $\Phi$, for each $X\in\cs$, we let $\Phi^X$ be the maximal $y\in\str\cup\cs$ in the order given by $\preceq$  such that $\Phi^{X \uh n}$ is a prefix of~$y$ for all~$n$. We will thus set
$\dom(\Phi)=\{X\in\cs:\Phi^X\in\cs\}$.  When $\Phi^X\in\cs$, we will often write $\Phi^X$ as $\Phi(X)$ to emphasize the functional $\Phi$ as a map from~$\cs$ to~$\cs$.  We also use the notation $\Phi^X \uh n \downarrow$ to emphasize that $\Phi^X$ has length at least~$n$. 
For $\tau \in \str$ let $\Phi^{-1}(\tau)$ be the set 
$\{ \sigma\in \str : \exists \tau' \succeq \tau \colon (\sigma,\tau')\in S_\Phi\}$.  In particular, by our above convention, we have $\Lambda\in\Phi^{-1}(\Lambda)$. Similarly, for $S \subseteq \str$ we define $\Phi^{-1}(S) = \bigcup_{\tau \in S} \Phi^{-1}(\tau)$. When $\mathcal{A}$ is a subset of $\cs$, we denote by $\Phi^{-1}(\mathcal{A})$ the set $\{X\in \dom(\Phi):\Phi(X) \in \mathcal{A}\}$. Note in particular that $\Phi^{-1}(\llb\tau\rrb)
= \llb\Phi^{-1}(\tau)\rrb \cap \dom(\Phi)$.\\

\begin{remark}\label{rmk-functionals-measures}
The Turing functionals that induce computable measures are precisely the \emph{almost total} Turing functionals, where a Turing functional $\Phi$ is almost total if 
\[
\lambda(\mathsf{dom}(\Phi))=1.
\]
Given an almost total Turing functional $\Phi$, the measure induced by $\Phi$, denoted $\lambda_\Phi$, is defined by
\[
\lambda_\Phi(\sigma)=\lambda(\llb\Phi^{-1}(\sigma)\rrb)
=\lambda(\{X:\Phi^X\succeq\sigma\}).
\]
It is not difficult to verify that $\lambda_\Phi$ is a computable measure.  Moreover, one can easily show that given a computable probability measure $\mu$, there is some almost total functional~$\Phi$ such that $\mu=\lambda_\Phi$.
\end{remark}

\subsection{Left-c.e.\ semi-measures}\label{subsec:semi-measures}

We consider semi-measures on both $\str$ and $\cs$.  In general, semi-measures behave like defective probability measures, as they need not be additive.

A \emph{discrete semi-measure} is a map $m:\str\rightarrow[0,1]$ such that $\sum_{\sigma\in\str}m(\sigma)\leq 1$. 
Moreover, if $S$ is a set of strings, then $m(S)$ is defined to be $\sum_{\sigma \in S} m(\sigma)$.  Henceforth, we will restrict our attention to the class of left-c.e.\ discrete semi-measures, where a discrete semi-measure $m$ is \emph{left-c.e.}\ if there is a computable function $\widetilde m : \str \times \N \rightarrow \binrat$, non-decreasing in its second argument such that for all~$\sigma$:
\[
\lim_{i \rightarrow +\infty} \widetilde m(\sigma,i) = m(\sigma).
\]
Levin showed that there is a \emph{universal} left-c.e.\ discrete semi-measure $\m$; that is, for every left-c.e.\ discrete semi-measure $m$, there is some constant $c$ such that $m\leq c\cdot \m$.

This universal discrete semi-measure $\m$ is closely related to the notion of prefix-free Kolmogorov complexity.
Recall that $\K(\sigma)$ denotes the prefix-free Kolmogorov complexity of~$\sigma$, i.e.\
\[
\K(\sigma)=\min\{|\tau|:U(\tau)=\sigma\},
\]
where $U$ is a universal prefix-free Turing machine.  Then by the coding theorem (see \cite[Theorem 3.9.4]{DowneyH2010}), we have $\K(\sigma)=-\log \m(\sigma)+O(1)$.  In particular, since $\K(n)\leq^+ 2\log(n)$, where $\K(n)$ is simply $\K(1^n)$, it follows that $\m(n)\geq^\times n^{-2}$, a fact we will make use of below.

A \emph{continuous semi-measure} is a map $\rho:\str\rightarrow[0,1]$ satisfying:
\begin{itemize}
\item[(i)] $\rho(\emptystring) = 1$ and 
\item[(ii)] $\rho(\sigma)\geq\rho(\sigma0)+\rho(\sigma1)$.
\end{itemize}  
This definition is closely related to the definition of a supermartingale (as defined in \cite{Nies2009} or \cite{DowneyH2010}).

If $S$ is a set of strings, then $\rho(S)$ denotes the sum $\sum_{\sigma \in S} \rho(\sigma)$.  As in the case of discrete semi-measures, we will restrict our attention to the class of left-c.e.\ continuous semi-measures (the values of which are effectively approximable from below as defined above in the case of discrete left-c.e.\ semi-measures).

One particularly important property of left-c.e.\ continuous semi-measures is their connection to Turing functionals.  Just as computable measures are precisely the measures that are induced by almost total Turing functionals (as discussed at the end of the previous section), left-c.e.\ continuous semi-measures are precisely the continuous semi-measures that are induced by Turing functionals:

\begin{theorem}[Levin, Zvonkin \cite{LevinZ1970}]\label{thm-MachinesInduceSemiMeasures}{\ }
\begin{itemize}
\item[(i)] For every Turing functional $\Phi$, the function $\lambda_\Phi(\sigma)=\lambda(\llb\Phi^{-1}(\sigma)\rrb)
=\lambda(\{X:\Phi^X\succeq\sigma\})$ is a left-c.e.\ continuous semi-measure.
\item[(ii)] For every left-c.e.\ continuous semi-measure $\rho$, there is a Turing functional $\Phi$ such that $\rho=\lambda_\Phi$. 
\end{itemize}
\end{theorem}

As there is a universal left-c.e.\ discrete semi-measure, so too is there a universal left-c.e.\ continuous semi-measure.  That is, there exists a left-c.e.\ continuous semi-measure $\M$ such that, for every left-c.e.\ continuous semi-measure~$\rho$, there exists a $c\in\N$ such that $\rho\leq c\cdot \M$.  One way to obtain a universal left-c.e.\ continuous semi-measure is to effectively list all left-c.e.\ continuous semi-measures $(\rho_e)_{e \in \N}$ (which can be obtained from an effective list of all Turing functionals by appealing to Theorem~\ref{thm-MachinesInduceSemiMeasures}) and set $\M = \sum_{e \in \N} 2^{-e-1} \rho_e$.  Alternatively, one can induce it by means of a universal Turing functional: Let $(\Phi_i)_{i\in\N}$ be an effective enumeration of all Turing functionals.  Then the functional $\widehat\Phi$ such that
\[
\widehat\Phi(1^e0X)=\Phi_e(X)
\]
\label{page:phi-hat}for every $e\in\N$ and $X\in\cs$ is a universal Turing functional and we can set $\M=\lambda_{\widehat\Phi}$. One can readily verify that $\M$ is a universal left-c.e.\ continuous semi-measure, which is sometimes called~\emph{a priori probability} (see for example G\'acs~\cite{Gacs-notes} for the basic properties of~$\M$).

Another feature of continuous semi-measures that we will make use of throughout this study is that there is a canonical measure on $\cs$ that can be obtained from a continuous semi-measure.
To motivate the definition of this measure, it is helpful to think of a semi-measure as a network flow through the full binary tree $\str$ seen as a directed graph (see, for instance, \cite{LevinV1977} or \cite{Vyugin1982}).  First we give the node at the root of the tree flow equal to $1$ (corresponding to the condition $\rho(\emptystring)=1$). Some amount of this flow at each node $\sigma$ is passed along to the node corresponding to $\sigma0$, some is passed along to the node corresponding to $\sigma1$, and, potentially, some of the flow is lost (corresponding to the condition that $\rho(\sigma)\geq\rho(\sigma0)+\rho(\sigma1)$).  We obtain a measure $\overline\rho$ from $\rho$ if we ignore all of the flow that is lost below a given node and just consider the behavior of the flow that never leaves the network below this node. We will refer to $\overline\rho$ as the canonical measure derived from $\rho$.  This can be formalized as follows.

\begin{definition}
Let $\rho$ be a semi-measure.  The \emph{canonical measure obtained from $\rho$} is defined to be 
\[
\overline\rho(\sigma):=\inf_{n\geq |\sigma|}\sum_{\tau\succeq\sigma\;\&\;|\tau|=n}\rho(\tau) = \lim_{n\rightarrow \infty} \sum_{\tau\succeq\sigma\;\&\;|\tau|=n}\rho(\tau).
\]
\end{definition}

Several important facts about these canonical measures are the following:

\begin{proposition}\label{prop-rhobar-properties}
Let $\rho$ be a semi-measure and let $\overline{\rho}$ be the canonical measure obtained from~$\rho$.
\begin{itemize}
\item[(i)] $\overline{\rho}$ is the largest measure $\mu$ such that $\mu\leq \rho$.  
\item[(ii)] If $\rho(\sigma)=\lambda(\{X:\Phi^X\succeq\sigma\})$,
then $\overline\rho(\sigma)=\lambda(\{X \in \dom(\Phi): \Phi^X\succeq\sigma\})$.
\end{itemize}
\end{proposition}

Thus, in replacing $\rho$ with $\overline\rho$, this amounts to restricting the Turing functional~$\Phi$ that induces~$\rho$ to those inputs on which~$\Phi$ is total.  The proof of (i) is straightforward; for a proof of (ii), see \cite{BienvenuHPS2014}.

\begin{remark}\label{rmk-Mbar}
Using Proposition \ref{prop-rhobar-properties}(ii) and the universality of $\M$, one can readily verify that for every left-c.e.\ continuous semi-measure $\rho$, there is some $c\in\N$ such that $\overline\rho \leq c\cdot\Mb$.  Thus $\Mb$ can be seen as a measure that is universal for the class of canonical measures obtained from some left-c.e.\ continuous semi-measure (a class that contains all computable measures).
\end{remark}

\subsection{Notions of algorithmic randomness}

The primary notion of algorithmic randomness that we will consider here is Martin-L\"of randomness.  Although the standard definition is given in terms of the Lebesgue measure, we will consider Martin-L\"of randomness with respect to any computable measure.

\begin{definition}
Let $\mu$ be a computable measure on $\cs$.

\begin{itemize}
\item[(i)] A \emph{$\mu$-Martin-L\"of test} is a sequence $\{\mathcal{U}_i\}_{i\in\N}$ of uniformly effectively open subsets of $\cs$ such that for each $i$,
\[
\mu(\mathcal{U}_i)\leq 2^{-i}.
\]
\item[(ii)] $X\in\cs$ passes the $\mu$-Martin-L\"of test $(\mathcal{U}_i)_{i\in\N}$ if $X\notin\bigcap_{i \in \N}\mathcal{U}_i$.
\item[(iii)] $X\in\cs$ is \emph{$\mu$-Martin-L\"of random}, denoted $X\in\MLR_\mu$, if $X$ passes every $\mu$-Martin-L\"of test. \end{itemize}
\end{definition}

When $\mu$ is the uniform (or Lebesgue) measure $\lambda$, we will simply write $\MLR$.  An important feature of Martin-L\"of randomness is the existence of a universal test: For every computable measure $\mu$, there is a single $\mu$-Martin-L\"of test $\{\hat{\mathcal{U}}_i\}_{i\in\N}$, having the property that $X\in\MLR_\mu$ if and only if $X\notin\bigcap_{i\in\N}\hat{\mathcal{U}}_i$.

\begin{remark}\label{rmk-atom}
As we saw earlier, some computable measures have atoms, such as the Dirac measure $\delta_X$ concentrated on some computable sequence $X$.  Moreover, given a computable measure $\mu$, if $X$ is a $\mu$-atom, it immediately follows that $X\in\MLR_\mu$.
\end{remark}

Four additional notions of algorithmic randomness that will be considered in this study are difference randomness, Kurtz randomness, weak 2-randomness, and computable randomness.

\begin{definition}
\begin{itemize}
\item[(i)] A \emph{difference test} is a computable sequence $\{(\mathcal{U}_i,\mathcal{V}_i)\}_{i\in\N}$ of pairs of $\Sigma^0_1$ classes such that for each $i$,
\[
\lambda(\mathcal{U}_i\setminus\mathcal{V}_i)\leq 2^{-i}.
\]
\item[(ii)] A sequence $X\in\cs$ \emph{passes a difference test} $\{(\mathcal{U}_i,\mathcal{V}_i)\}_{i\in\N}$ if $X\notin\bigcap_i(\mathcal{U}_i\setminus\mathcal{V}_i)$.
\item[(iii)] $X\in\cs$ is \emph{difference random} if $X$ passes every difference test. We denote by~$\DiffR$ the class of difference random reals. 
\end{itemize}

\end{definition}

Franklin and Ng proved the following remarkable theorem about difference randomness:

\begin{theorem}[Franklin and Ng~\cite{FranklinN2011}]\label{thm:franklin-ng}
A sequence~$X$ is difference random if and only if $X$ is Martin-L\"of random and $X \not\geq_T \emptyset'$. 
\end{theorem}

Recall that a sequence $X$ has \emph{$\PA$ degree} if $X$ can compute a consistent completion of Peano arithmetic (hereafter, $\PA$).  This is equivalent to requiring that $X$ computes a total function extending a universal partial computable $\{0,1\}$-valued function, a fact that will be useful in Section \ref{sec:examples}.  A related result is the following:

\begin{theorem}  [Stephan \cite{Stephan2002}]\label{thm:diff-pa}
A Martin-L\"of random sequence $X$ has $\PA$ degree if and only if $X\geq_T\emptyset'$.
\end{theorem}

\noindent It follows from the previous two results that a Martin-L\"of random sequence is difference random if and only if it does not have $\PA$ degree.

Another approach to defining randomness is to require that a random sequence avoid all null sets that are definable at some fixed level of syntactic complexity.  For instance, if we take all $\Pi^0_1$ definable null sets or all $\Pi^0_2$ definable null sets, we have the following two notions of randomness, first introduced by Kurtz in \cite{Kurtz1981}.

\begin{definition} Let $X\in\cs$.
\begin{itemize}
\item[(i)] $X$ is \emph{Kurtz random} (or \ \emph{weakly 1-random}) if and only if $X$ is not contained in any $\Pi^0_1$ class of Lebesgue measure~$0$ (equivalently, if and only if it is not contained in any $\Sigma^0_2$ class of measure~$0$).
\item[(ii)] $X$ is \emph{weakly 2-random} if and only if $X$ is not contained in any $\Pi^0_2$ class of Lebesgue measure~$0$.
\end{itemize}
\end{definition}
Let $\KR$ denoted the collection of Kurtz random sequences and $\WtwoR$ denote the collection of weakly 2-random sequences.

The last definition of randomness we will consider in the study is defined in terms of certain effective betting strategies called \emph{martingales}.

\begin{definition}
\begin{itemize}
\item[(i)] A \emph{martingale} is a function $d:\str\rightarrow\mathbb{R}^{\geq 0}$ such that for every $\sigma\in\str$,
\[2d(\sigma)=d(\sigma0)+d(\sigma1).\] 
\item[(ii)] A  martingale $d$ succeeds on $X\in\cs$ if 
\[\limsup_{n\rightarrow\infty}\, d(X\uh n)=+\infty.\]
\item[(iii)] A sequence $X\in\cs$ is \emph{computably random} if there is no computable martingale $d$ that succeeds on $X$. 
\end{itemize}
\end{definition}

\noindent The collection of computably random sequences will be written as $\CR$. The different notions of randomness discussed in this section form a strict hierarchy. Namely, the following relations hold: 
\[
\WtwoR\subsetneq\DiffR\subsetneq\MLR\subsetneq\CR\subsetneq\KR
\]
We note that each of these notions of randomness can be defined with respect to a computable measure $\mu$, similarly to how we defined $\mu$-Martin-L\"of randomness.  Moreover, each of the above notions can be relativized to an oracle $A\in\cs$ in a straightforward manner.
\section{Negligibility and $\tt$-Negligibility}\label{sec:neg}

We are now in a position to define negligibility and $\tt$-negligibility. As discussed in the introduction, the intuitive idea behind negligibility is that a set $\C\subseteq\cs$ is negligible if no member of $\C$ can be produced with positive probability by means of any Turing functional with a random oracle.  Similarly, a set $\C\subseteq\cs$ is $\tt$-negligible if no member of $\C$ can be produced with positive probability by means of any \emph{total} Turing functional with a random oracle.  However, we will primarily work with the following measure-theoretic definition of these two notions.

\begin{definition} Let $\C\subseteq\cs$. 
\begin{itemize}
\item[(i)] $\C$ is \emph{negligible} if $\Mb(\C)=0$. 
\item[(ii)] $\C$ is \emph{tt-negligible} if $\mu(\C)
=0$ for every computable measure~$\mu$.
\end{itemize}
 \end{definition}
 
The intuitive description of negligibility and $\tt$-negligibility given above is justified by the following proposition. For $\C\subseteq\cs$, we define $\C^{\leq_T}$ to be the set $\{Y \in \cs : (\exists X\in\C)[X\leq_T Y]\}$.  Furthermore, $\C^{\leq_{tt}}$ denotes the set $\{Y \in \cs : (\exists X\in\C)[X\leq_{tt} Y]\}.$

\begin{proposition}\label{prop:carac_neglible} Let $\C\subseteq\cs$.
\begin{itemize}
\item[(i)] $\C$ is negligible if and only if $\lambda(\C^{\leq_T})=0$.
\item[(ii)] $\C$ is $\tt$-negligible if and only if $\lambda(\C^{\leq_{tt}})=0$.
\end{itemize}
\end{proposition}

\begin{proof} (i)  ($\Rightarrow$)  If $\C$ is negligible, then $\overline\rho(\C)=0$ for every left-c.e.\ continuous semi-measure $\rho$ by Remark \ref{rmk-Mbar}.  In particular, by Theorem \ref{thm-MachinesInduceSemiMeasures}(ii) and Proposition \ref{prop-rhobar-properties}(ii), $\overline\lambda_\Phi(\C)=0$ for every Turing functional $\Phi$ and hence
\[
\lambda(\C^{\leq_T})=\sum_{i\in\N}\lambda(\{Y \in \cs : \Phi_i(Y)\in\C\})=\sum_{i\in\N}\overline\lambda_{\Phi_i}(\C)=0.
\]
($\Leftarrow$) Since $\lambda(\C^{\leq_T})=0$, it follows that $\overline\lambda_\Phi(\C)=0$, where $\Phi$ is a universal Turing functional.  Thus $\Mb(\C)=\overline\lambda_\Phi(\C)=0$.

\bigskip
\noindent (ii)  ($\Rightarrow$)  If $\C$ is $\tt$-negligible, then $\lambda_\Phi(\C)=0$ for every $\tt$-functional $\Phi$.  The result clearly follows.\\
($\Leftarrow$) Now suppose that $\mu(\C)>0$ for some computable measure $\mu$.  Then by Remark \ref{rmk-functionals-measures}, there is some almost total Turing functional $\Phi$ such that $\lambda_\Phi=\mu$.  Since the domain of a Turing functional is $\Pi^0_2$, if $X\notin\dom(\Phi)$, it follows that $X\notin\MLR$ (in fact, $X$ is not even Kurtz random).  Indeed, the complement of $\dom(\Phi)$ is a $\Sigma^0_2$ class of measure~$0$, so $X\notin\dom(\Phi)$ implies that $X$ is contained in a $\Pi^0_1$ class of measure~$0$.

Let $i$ be the least such that $\mu(\C)>2^{-i}$.  Then $\lambda(\Phi^{-1}(\C)\cap\hat\U^c_i)>0$, where $\hat\U_i$ is the $i^{\mathrm{th}}$ level of the universal Martin-L\"of test (so that $\lambda(\hat\U_i^c)>1-2^{-i}$). Now, since $\Phi$ is total on $\hat\U_i^c$ (because $\Phi$ is defined on all Martin-L\"of randoms, and $\hat\U_i^c$ contains only Martin-L\"of randoms), which is a $\Pi^0_1$ class, there is a $\tt$-functional $\Psi$ which coincides with $\Phi$ on $\hat\U_i^c$.  This holds because for every $\Pi^0_1$ class~$\P$ and every Turing functional~$\Phi$ that is total on $\P$, there is a Turing functional~$\Psi$ that is total on $\cs$ and coincides with $\Phi$ on~$\P$.

Since $\Phi^{-1}(\C)\cap\hat\U^c_i=\Psi^{-1}(\C)\cap\hat\U^c_i$, it follows that $\lambda(\Psi^{-1}(\C)\cap\hat\U^c_i)>0$ and hence $\lambda(\C^{\leq_{tt}})>0$.


\end{proof}

By the following proposition, there is a simple characterization of negligible and $\tt$-negligible singletons.

\begin{proposition} For $X\in\cs$, the following are equivalent:
\begin{itemize}
\item[(i)] $\{X\}$ is negligible.
\item[(ii)] $\{X\}$ is $\tt$-negligible.
\item[(iii)] $X$ is non-computable.
\end{itemize}
\end{proposition}

\begin{proof}
(i)$\Rightarrow$(ii) is immediate.  (ii)$\Rightarrow$(iii) follows from Lemma \ref{lem-kautz}, which tells us that $X$ is computable if $\mu(\{X\})>0$ for some computable measure $\mu$.  Lastly, (iii)$\Rightarrow$(i) is the theorem of De Leeuw et al.\ that was mentioned in the introduction.  
\end{proof}

Further, it is clear that for a $\Sigma^0_1$ class $\sS\subseteq\cs$, $\sS$ is negligible if and only if $\sS$ is $\tt$-negligible if and only if $\sS$ is empty.  However, the notions of negligibility and $\tt$-negligibility are non-trivial for $\Pi^0_1$ classes.  For instance, if we let $\sPA$ be the $\Pi^0_1$ class of consistent completions of $\PA$, then as shown by Jockusch and Soare in~\cite{JockuschS1972}, $\sPA$ is negligible (and thus $tt$-negligible).


Although every negligible $\Pi^0_1$ class is $\tt$-negligible, the converse does not hold.

\begin{theorem}\label{thm:tt-negligible-not-negligible}
There exists a tt-negligible $\Pi^0_1$ class that is not negligible. 
\end{theorem}

The proof of Theorem~\ref{thm:tt-negligible-not-negligible} draws upon a theorem of Downey, Greenberg and Miller~\cite{DowneyGM2008}: there exists a non-negligible perfect thin $\Pi^0_1$ class, where a $\Pi^0_1$ class~$\C$ is \emph{thin} if for every $\Pi^0_1$ subclass $\C'$ of $\C$, there exists a clopen set $D$ such that $\C' = \C \cap D$. We argue that any perfect thin class must be $\tt$-negligible, which extends a result of Simpson~\cite{Simpson2005} who proved that every thin class must have Lebesgue measure~$0$.  However, our proof strategy is very different from Simpson's.  We first need the following lemma, which is folklore. 

\begin{lemma}\label{lem:comp-measure-member}
Let $\C$ be a $\Pi^0_1$ class. If for some computable probability measure~$\mu$ the value of $\mu(\C)$ is a positive computable real number, then $\C$ contains a computable member. 
\end{lemma}

\begin{proof}
Consider the measure $\nu: \str \rightarrow [0,1]$ defined by $\nu(\sigma)=\mu(\C \cap \llb\sigma\rrb)$. Since $\C$ is $\Pi^0_1$, the measure~$\nu$ is right-c.e.\ (that is, the values $\nu(\sigma)$ for $\sigma\in\str$ are uniformly right-c.e.). However, for every $\sigma$, we have
\[
\nu(\sigma)=\nu(\emptystring)-\sum_{\substack{|\tau|=|\sigma|\\\tau \not=\sigma}} \nu(\tau)
\]
and since $\nu(\emptystring)$ is computable (it is equal to $\mu(\C)$), this shows that $\nu(\sigma)$ is also left-c.e.\ uniformly in~$\sigma$. Therefore $\nu$ is computable. It is then easy to computably build by induction a sequence of strings $\sigma_0 \prec \sigma_1 \prec \ldots$ with $|\sigma_i|=i$ and such that $\nu(\sigma_i) \geq \nu(\emptystring) \cdot 4^{-i} >0$. In particular $\C \cap \llb\sigma_i\rrb \not= \emptyset$. Therefore, the sequence $X$ extending all~$\sigma_i$ is computable and must be an element of $\C$.
\end{proof}

\begin{lemma}
If $\C$ is a perfect thin $\Pi^0_1$ class, then it is $\tt$-negligible. 
\end{lemma}

\begin{proof}  First, observe that if a thin class $\P$ contains a computable member~$X$, then since $\{X\}$ is a $\Pi^0_1$ subclass of $\P$, there is some $\sigma\in\str$ such that $\{X\}=\P\cap\llb\sigma\rrb$.  Thus, $X$ is isolated in $\P$.  It thus follows that a perfect thin $\Pi^0_1$ class contains no computable members.  

Now, for the sake of contradiction, suppose that there exists a computable measure~$\mu$ such that $\mu(\C)>q$ for some positive rational~$q$. Identifying $\cs$ with the unit interval $[0,1]$ in the usual way, let
\[
\alpha = \sup\big\{r \in [0,1] \cap \Q \, : \, \mu(\C \cap [0,r]) < q \big\}.
\]
Since $ \mu(\C \cap [0,r]) < q$ is a $\Sigma^0_1$ predicate in $r$, $\alpha$ is a left-c.e.\ real. Moreover, since~$\C$ has no computable member, it contains no atom of $\mu$, and thus the function $x \mapsto \mu(\C \cap [0,x])$ is continuous. Therefore, by definition of $\alpha$, $\mu(\C \cap [0,\alpha))=\mu(\C \cap [0,\alpha])=q$. \\

Consider the class $\C \cap [\alpha,1]$, which is a $\Pi^0_1$ subclass of $\C$ since $\alpha$ is left-c.e. By the thinness of $\C$, there exists a clopen set~$D$ such that $\C \cap [\alpha,1]=\C \cap D$ and thus such that $\C \cap D^c=\C \cap [0,\alpha)$. From this we deduce $\mu(\C \cap D^c)=q$, which is a positive, computable real number. Since $D$ is clopen, $\C \cap D^c$ is $\Pi^0_1$, and so by Lemma~\ref{lem:comp-measure-member}, $\C \cap D^c$ contains a computable member. This contradicts our above observation that $\C$ contains no such members.\\

%
%
\end{proof}

\section{Depth and $\tt$-depth}\label{sec:depth}

Members of negligible and $\tt$-negligible classes are difficult to produce, in the sense that their members cannot be computed from random oracles with positive probability.  Given a $\Pi^0_1$ class $\P$, we can instead consider the probability of producing an \emph{initial segment} of some member of $\P$.  By looking at local versions of negligibility and $\tt$-negligibility, given in terms of initial segments of members of the classes in question, we will obtain the notions of depth and $\tt$-depth, respectively. Since we are interested in strings that are initial segments of some member of $\C$, the natural representation of~$\C$ as a tree is via its co-c.e.\ tree (recall that it is precisely the tree consisting of the strings~$\sigma$ that are prefixes of some path of $\C$, and by the standard compactness argument, this is a co-c.e.\ set of strings).


Recall our convention from Section \ref{subsec:semi-measures}: if $S$ is a set of strings, then $\M(S)$ and $\m(S)$ are defined to be $\sum_{\sigma \in S} \M(\sigma)$ and $\sum_{\sigma \in S} \m(\sigma)$, respectively. Similarly, if $\mu$ is a measure, $\mu(S)$ denotes $\sum_{\sigma \in S} \mu(\sigma)$.  Furthermore, recall from Section \ref{subsec:notation} that given a tree $T\subseteq\str$, $T_n$ denotes the set of all members of $T$ of length $n$. We now define the central notion of this paper, namely the notion of a deep $\Pi^0_1$ class.  Depth strengthens the notion of negligibility. It is easy to see that a class $\C\subseteq\cs$ of canonical co-c.e.\ tree~$T$ is negligible if and only if $\M(T_n)$ converges to~$0$ as $n$ grows without bound.  When this convergence to~$0$ is effective, $\C$ is said to be deep.

\begin{definition}
Let $\C\subseteq\cs$ be a $\Pi^0_1$ class and $T$ its associated co-c.e.\ tree. 
\begin{itemize}
\item[(i)] $\C$ is \emph{deep} if there is some computable order function~$h$ such that for all~$n$,
\[
\M(T_n) \leq 2^{-h(n)}
\] 
\item[(ii)] $\C$ is \emph{$\tt$-deep} if for every computable measure $\mu$ there exists a computable order function~$h$ such that for all~$n$,
\[
\mu(T_n) \leq 2^{-h(n)}.
\]
\end{itemize}
\end{definition}

Our choice of the term ``deep" is due to the similarity between the notions of a deep class and that of a logically deep sequence, introduced by Bennett in \cite{Bennett1995}.  Roughly, a sequence $X$ is logically deep if there is no computable function that bounds the amount of time to recover each initial segment of $X$ from its shortest description (measured in terms of Kolmogov complexity).  Logically deep sequences are highly structured; in particular, it is difficult to produce initial segments of a deep sequence via probabilistic computation, a fact made precise for finite strings by what Bennett calls the slow growth law (see \cite[Theorem~1]{Bennett1995}).   Deep classes can thus be seen as an analogue for $\Pi^0_1$ classes of logically deep sequences.

\begin{remark}\label{rem:deep-order}
Note that depth can be equivalently defined as follows: $\C$ is deep if for some computable function~$f$ one has $\M(T_{f(k)}) \leq 2^{-k}$ for all~$k$. Indeed, if $\M(T_{n})\leq 2^{-h(n)}$, then, setting $f=h^{-1}$, we have $\M(T_{f(k)}) \leq 2^{-n}$, and conversely, if $\M(T_{f(n)}) \leq 2^{-n}$, we can assume that~$f$ is increasing and then taking $h=f^{-1}$, we have $\M(T_n)\leq 2^{-h(n)}$ (here we use the fact that $\M(T_n)$ is non-increasing in $n$).

The same argument shows that $\C$ is $\tt$-deep if and only if for every computable measure~$\mu$, there exists a computable function~$f$ such that for all~$n$, $\mu(T_{f(n)}) \leq 2^{-n}$.  \\
\end{remark}

\begin{remark}\label{rem:deep-with-m}
Another alternative way to define depth is to use $\m$ instead of~$\M$: a $\Pi^0_1$ class~$\C$ with canonical co-c.e.\ tree~$T$ is deep if and only if there is a computable order function~$h$ such that $\m(T_n) \leq 2^{-h(n)}$ for all~$n$, if and only if there is a computable function~$f$ such that $\m(T_{f(n)})\leq 2^{-n}$ for all~$n$. Indeed, the following inequality holds for all strings~$\sigma$ (see for example \cite{Gacs-notes}):
\[
\m(\sigma)  \leq^\times \M(\sigma) \leq^\times \m(\sigma)/ \m(|\sigma|) 
\]
Thus, for~$h$ a computable order function, if $\M(T_n) \leq 2^{-h(n)}$ for all~$n$, then there is some $c$ such that $\m(T_n) \leq 2^{-h(n)+c}$ for all~$n$. Conversely, if $\m(T_n) \leq 2^{-h(n)}$, then $\m(T_{f(n)}) \leq 2^{-n}$ for $f=h^{-1}$, and by the above inequality,
\[
\M(T_{f(n)}) \leq^\times 2^{-n}/ \m(f(n)) \leq^\times 2^{-n} \cdot n^2. 
\]
Thus, taking $g(n)=f(2n+c)$ for some large enough constant~$c$, we get~$\M(T_{g(n)})\leq 2^{-n}$, and thus $\C$ is deep. \\

\end{remark}

Since every $\Pi^0_1$ class $\C$ is the set of paths through a computable tree, why can't we simply define depth and $\tt$-depth in terms of this tree and not the canonical co-c.e.\ tree associated to $\C$?  In the case of depth, there are two reasons to restrict to the canonical co-c.e.\ trees associated to $\Pi^0_1$ classes.

First, the idea behind a deep class is that it is difficult to produce initial segments of some member of the class.  In general, for a $\Pi^0_1$ class $\C$, any computable tree $T$ contains non-extendible nodes, and so if we have a procedure that can compute these non-extendible nodes of $T$ with high probability, this tells us nothing about the difficulty of computing the extendible nodes of~$T$. 

Second, if we were to use any tree~$T$ representing a $\Pi^0_1$ class~$\C$ in the definition of depth, then depth would become a void notion, by the following proposition.

\begin{proposition}\label{prop:no-deep-comp-tree}
If $T$ is an infinite computable tree, then there is no computable order function~$f$ such that $\m(T_{f(n)}) \geq 2^{-n}$. 
\end{proposition}

\begin{proof}
Given an infinite computable tree~$T$, let $f$ be a computable order function.  Then there is a computable sequence~$(\sigma_n)_{n \in \N}$ of strings such that $\sigma_n \in T_{f(n)}$ for every $n$. Thus 
\[
\m(T_{f(n)}) \geq \m(\sigma_n) \geq^\times \m(n) \geq^\times 1/n^2.
\] 
\end{proof}

By contrast with the notion of depth, for the definition of $\tt$-depth, it does not matter whether we work with the canonical co-c.e.\ tree associated to $\C$ or some computable tree $T$ such that $\C=[T]$ (a direct consequence of Theorem~\ref{prop:carac-tt-negl}). 

An important question concerns the relationship between depth and negligibility (and between $\tt$-depth and $\tt$-negligibility).  First, every deep  $\Pi^0_1$ class is negligible.  Indeed, if $\C$ is deep, $T$ is the canonical co-c.e.\ tree associated to $\C$, and $
\M(T_n) \leq 2^{-h(n)}$ for some computable order function $h$, then
\[
\Mb(\C)=\Mb\bigl(\bigcap_{n\in\N}\llb T_n\rrb\bigr)=\lim_{n\rightarrow\infty}\Mb(T_n)\leq\lim_{n\rightarrow\infty}\M(T_n)=0,
\]
where the second equality follows from the continuity of $\Mb$ from above.  A similar argument shows that $\tt$-depth implies $\tt$-negligibility.  Does the converse hold?  In the case of $\tt$-depth and $\tt$-negligibility, the answer is positive.  We also identify two other equivalent formulations of $\tt$-depth.  

\begin{theorem} \label{prop:carac-tt-negl}
Let $\C$ be a $\Pi^0_1$ class. The following are equivalent:
\begin{itemize}
\item[(i)] $\C$ is $\tt$-deep.
\item[(ii)] $\C$ is $\tt$-negligible. 
\item[(iii)] For every computable measure~$\mu$, $\C$ contains no $\mu$-Kurtz random element. 
\item[(iv)] For every computable measure~$\mu$, $\C$ contains no $\mu$-Martin-L\"of random element. 
\end{itemize}
\end{theorem}

\begin{proof}
(i)$\Rightarrow$(ii). This is shown by a direct modification of the above proof that depth implies negligibility.\\
(ii)$\Rightarrow$(iii). This follows directly from the definition of Kurtz randomness.\\
(iii) $\Rightarrow$ (iv). This follows from the fact that Martin-L\"of randomness implies Kurtz randomness.\\
(iv) $\Rightarrow$ (ii). If $\mu(\C)>0$ for some computable measure $\mu$, then $\C$ must contain some $\mu$-Martin-L\"of random element since the set of $\mu$-Martin-L\"of random sequences has $\mu$-measure~$1$. \\
(ii)$\Rightarrow$ (i). Suppose that $\C$ is $\tt$-negligible and let $\mu$ be a computable measure. Let~$T$ be the co-c.e.\ tree associated to~$\C$. By $\tt$-negligibility, $\mu(\C)=0$, or equivalently, $\mu(T_n)$ tends to~$0$. Since the $T_n$ are co-c.e.\ sets of strings, $\mu(T_n)$ is right-c.e.\ uniformly in~$n$. Thus, given~$k$, it is possible to effectively find an~$n$ such that $\mu(T_n) \leq 2^{-k}$. Setting $f(k)=n$, we have a computable function~$f$ such that $\mu(T_{f(k)})\leq2^{-k}$ for all~$k$, therefore $\C$ is $\tt$-deep by Remark~\ref{rem:deep-order}. 
\end{proof}

\begin{remark}
Recall that a $\Pi^0_1$ class is \emph{special} if it has no computable member.  Given a computable sequence $X$, by Remark \ref{rmk-atom}, there is a computable measure $\mu$ such that $X\in\MLR_\mu$.  It thus
follows from Theorem \ref{prop:carac-tt-negl} that $\tt$-negligible classes contain no computable members and are thus special.
\end{remark}

Significantly, in contrast to the case with $\tt$-negligibility and $\tt$-depth, negligibility and depth do not coincide. This is due to a fundamental aspect of depth, namely that it is not invariant under Turing equivalence, unlike negligibility. Suppose that $\C$ and $\D$ are two classes such that for every $X \in \C$ there exists a $Y \in \D$ such that $X \equiv_T Y$, and vice-versa. Then by Proposition~\ref{prop:carac_neglible}, $\C$ is negligible if and only if~$\D$ is negligible. This invariance does not hold in general for deep classes, as the next theorem shows.


\begin{theorem}\label{thm:depth-not-Turing}
For any $\Pi^0_1$ class~$\C$, there is a $\Pi^0_1$ class~$\D$ such that:
\begin{itemize}
\item[(a)] the elements of~$\D$ are the same as the elements of~$\C$, modulo deletion of a finite prefix (which in particular guarantees that the elements of~$\C$ and $\D$ have the same Turing degrees); and
\item[(b)] $\D$ is not deep.
\end{itemize}
\end{theorem}

\begin{proof}
Let $\C$ be a $\Pi^0_1$ class. Let $T$ be a computable tree such that $[T]=\C$, and let $S$ be the canonical co-c.e.\ tree associated to~$\C$. Consider the tree~$U$ obtained by appending a copy of~$S$ to each terminal node of~$T$. Formally, 
\[
U = T \cup \{\sigma^ \frown \tau \colon \sigma \in T_{\mathit{term}} ~\text{and}~ \tau \in S\},
\]
where~$T_{\mathit{term}}$ is the set of terminal nodes of~$T$, which is a computable set. One can readily verify that $U$ is the canonical co-c.e.\ tree associated to some $\Pi^0_1$ class. Indeed, suppose $\rho$ is a node of~$U$.  Then,
\begin{itemize}
\item either $\rho \in T$, in which case $\rho$ is either a prefix of some~$X\in[T]$ (and thus~$X\in[U]$), or $\rho$ is a prefix of a terminal node $\rho'$, which can then be extended to some~$(\rho')^\frown X\in[U]$, where $X\in[S]$; or 
\item $\rho$ is of the form $\sigma^\frown \tau$, with $\sigma \in T_{\mathit{term}}$ and $\tau \in S$. Since $S$ is a canonical co-c.e.\ tree, all its nodes extend to an infinite path, and thus $\sigma^\frown \tau$ has an extension $\sigma^\frown X\in[U]$ for some path~$X\in[S]$.
\end{itemize}

The fact that $U$ is co-c.e.\ follows directly from the fact that $T_{term}$ is computable and $S$ is co-c.e. Let~$\D$ be the $\Pi^0_1$ class whose canonical co-c.e.\ tree is~$U$. Then the elements of~$\D$ are either elements of~$[T]=\C$, or are of the form $\sigma^\frown X$ for some finite string~$\sigma$ and $X \in \C$, which gives us the first part of the conclusion. Finally, the canonical co-c.e.\ tree of $\D$ contains an infinite computable tree, namely, the tree~$T$, and therefore, by the argument as in the proof of Proposition~\ref{prop:no-deep-comp-tree}, $\D$ is not deep. 

\end{proof}

It is now straightforward to get a $\Pi^0_1$ class that is negligible but not deep: it suffices to take a negligible class~$\C$ and apply the above theorem. The resulting class~$\D$ is also negligible as its elements have the same Turing degrees as the elements of~$\C$ and negligibility is preserved under Turing equivalence by Proposition \ref{prop:carac_neglible}(i).  However, $\D$ is not deep. \\

The following summarizes the implications between the different properties introduced above for $\Pi^0_1$ classes.\\

\begin{center}
deep $\Rightarrow$ negligible $\Rightarrow$ $\tt$-negligible $\Leftrightarrow$ $\tt$-deep $\Rightarrow$ special
\end{center}
%

\section{Computational limits of randomness}\label{sec:comp-limits}

As we have seen, negligible $\Pi^0_1$ classes (and thus deep classes) have the property that one cannot compute a member of them with positive probability.  Although some random sequences can compute a member in any negligible $\Pi^0_1$ class (namely the Martin-L\"of random sequences of $\PA$ degree, as sequences of $\PA$ degree compute a member of \emph{every} $\Pi^0_1$ class), by the definition of negligibility, almost every random sequence fails to compute a member of a negligible class.

Similarly, one cannot $\tt$-compute a member of a $\tt$-negligible $\Pi^0_1$ class with positive probability.  The definition of $\tt$-negligibility implies that the collection of sequences that can $\tt$-compute a member of a $\tt$-negligible class form a set of Lebesgue measure zero.

In this section, we specify a precise level of randomness at which computing a member of a $\tt$-negligible, negligible, or deep class fails.  First, we consider the case for $\tt$-negligible classes.

\begin{theorem}
If $X\in\cs$ is Kurtz random, it cannot $\tt$-compute any member of a $\tt$-negligible $\Pi^0_1$ class. 
\end{theorem}

\begin{proof}
Let $\C$ be a $\tt$-negligible $\Pi^0_1$ class and $\Phi$ a $\tt$-functional. The set $\Phi^{-1}(\C)$ is a $\Pi^0_1$ class that, by $\tt$-negligibility, has Lebesgue measure~$0$. Thus, it contains no Kurtz random. 
\end{proof}

A similar proof can be used to prove the following:

\begin{theorem}
If $X\in\cs$ is weakly 2-random, it cannot compute any member of a negligible $\Pi^0_1$ class. 
\end{theorem}

\begin{proof}
Let $\C$ be negligible $\Pi^0_1$ class, $\Phi$ be a Turing functional, and $T$ a computable tree such that $\C=[T]$. The set $\Phi^{-1}(\C)$ is a $\Pi^0_2$ class, since
\[
X\in\Phi^{-1}(\C)\;\text{if and only if}\;(\forall k)(\exists \sigma)(\exists n)[\sigma\in T\;\&\;|\sigma|=k\;\&\; \sigma \preceq \Phi^{X \uh n}]
\]
By negligibility, $\Phi^{-1}(\C)$ has Lebesgue measure~$0$ and thus contains no weakly 2-random sequence. 
\end{proof}

We do not know whether this last theorem is optimal, i.e., whether it can be extended to randomness notions that are weaker than weak-2-randomness. \\

Our next result, despite its simplicity, is probably the most interesting of this section. It will help us unify a number of theorems that have appeared in the literature. These are theorems of form \\

\begin{center}
($*$) If $X$ is difference random, then it cannot compute an element of $\C$,
\end{center}
$ $\\
 where $\C$ is a given $\Pi^0_1$ class. Theorem~\ref{thm:diff-pa} is an example of such a theorem, with~$\C$ the class of consistent completions of $\PA$. The same result has been obtained with $\C$ the class of shift-complex sequences (Khan \cite{Khan2013}), the set of compression functions (Greenberg, Miller, Nies \cite{GreenbergMN-ip}), and the set $\DNC_q$ functions for some computable order functions~$q$ (Miller, unpublished). We will give the precise definition of these classes in Section \ref{sec:examples} but the important fact is that all of these classes are deep, and indeed, showing the depth of a $\Pi^0_1$ class is sufficient to obtain a theorem of the form ($*$). \label{page:discussion-diff-deep}
\begin{theorem}\label{thm:diff-deep}
If a sequence~$X$ is difference random, it cannot compute any member of a deep $\Pi^0_1$ class. 
\end{theorem}

\begin{proof}
Let $\C$ be a deep $\Pi^0_1$ class with associated co-c.e.\ tree~$T$ and let~$f$ be a computable function such that $\M(T_{f(n)}) \leq 2^{-n}$. Let~$X$ be a sequence that computes a member of $\C$ via a Turing functional $\Phi$. Let
\[
\mathcal{Z}_n = \{Z \, : \, \Phi^Z \uh f(n)  \downarrow  \, \in T_{f(n)}\}
\]
The set $\mathcal{Z}_n$ can be written as the difference $\U_n \setminus \V_n$ of two effectively open sets (uniformly in~$n$) with $\U_n = \{Z \, : \, \Phi^Z \uh f(n)  \downarrow\}$ and $\V_n= \{Z \, : \, \Phi^Z \uh f(n)  \downarrow  \notin T_{f(n)}\}$.

Moreover, by definition of the semi-measure induced by $\Phi$,
\[
\lambda(\mathcal{Z}_n) \leq \lambda_\Phi(T_{f(n)}) \leq^\times \M(T_{f(n)}) \leq 2^{-n}.
\]
The sequence $\{\mathcal{Z}_n\}_{n\in\N}$ thus yields a difference test.  Therefore, the sequence~$X$, which by assumption belongs to all~$\mathcal{Z}_n$, is not difference random.

\end{proof}

We remark that the converse does not hold: i.e., there is a class~$\D$ such that (i) $\D$ is not deep but (ii) no difference random real can compute an element of~$\D$. Indeed, take a deep class~$\C$ and apply Theorem~\ref{thm:depth-not-Turing} to get a class~$\D$ that is not deep but whose members have the same Turing degrees as the elements of~$\C$. Thus, no difference random real can compute an element of~$\D$. Also, the hypothesis of difference randomness cannot be improved. Indeed, if $X$ is difference random but not Martin-L\"of random, then it computes $\emptyset'$, and thus it can compute a member of any non-empty $\Pi^0_1$ class.

\section{Depth, negligibility, and mass problems}\label{sec:mass-problems}

In this section we discuss depth and negligibility in the context of mass problems, i.e., in the context of Muchnik and Medvedev reducibility.  Both Muchnik and Medvedev reducibility are generalizations of Turing reducibility.  Whereas Turing reducibility is defined in terms of a pair of sequences, both Muchnik and Medevedev reducibility are defined in terms of a pair of collections of sequences.  In what follows, we will consider these two reducibilities when restricted to $\Pi^0_1$ subclasses of $\cs$.  We will follow the notation of the survey~\cite{Simpson2011}, to which we refer the reader for a thorough exposition of mass problems in the context of $\Pi^0_1$ classes. 

Let $\C,\D\subseteq\cs$ be $\Pi^0_1$ classes.
We say that $\C$ is  \emph{Muchnik reducible} to $\D$, denoted $\C\leq_w\D$, if for every $X\in\D$ there exists a Turing functional $\Phi$ such that (i) $X\in\dom(\Phi)$ and $\Phi(X)\in\C$.  Moreover, $\C$  is  \emph{Medvedev reducible} to $\D$, denoted $\C\leq_s\D$, if $\C$ is Muchnik reducible to $\D$ via a single Turing functional, i.e., there exists a Turing functional $\Phi$ such that (i) $\D\subseteq\dom(\Phi)$ and $\Phi(\D)\subseteq\C$.  

Just as Turing reducibility gives rise to a degree structure, we can define degree structures from $\leq_w$ and $\leq_s$. We say that $\C$ and $\D$ are \emph{Muchnik equivalent} (resp.\ \emph{Medvedev equivalent}), denoted $\C\equiv_w\D$ (resp.\ $\C\equiv_s\D$) if and only if $\C\leq_w\D$ and $\D\leq_w\C$ (resp.\ $\C\leq_s\D$ and $\D\leq_s\C$).  The collections of Muchnik and Medvedev degrees given by the equivalence classes under $\equiv_w$ and $\equiv_s$ are denoted $\E_w$ and $\E_s$, respectively.

Both $\E_w$ and $\E_s$ are lattices, unlike the Turing degrees, which only form an upper semi-lattice.  We define the meet and join operations as follows.  Given $\Pi^0_1$ classes $\C,\D\subseteq\cs$, $\sup(\C,\D)$ is the $\Pi^0_1$ class $\{X\oplus Y:X\in\C\;\&\;Y\in\D\}$.  Furthermore, we define $\inf(\C,\D)$ to be the $\Pi^0_1$ class $\{0^\frown X: X\in\C\}\cup\{1^\frown Y:Y\in\D\}$.  One can readily check that the least upper bound of $\C$ and $\D$ in $\E_w$ is the Muchnik degree of $\sup(\C,\D)$ while their greatest lower bound is the Muchnik degree $\inf(\C,\D)$, and similarly for $\E_s$.

Recall that a filter $\F$ in a lattice $(\L,\leq,\inf,\sup)$ is a subset that satisfies the following two conditions:  (i) for all $x,y\in\L$, if $x\in\F$ and $x\leq y$, then $y\in\F$, and (ii) for all $x,y\in\F$, $\inf(x,y)\in\F$.



The goal of this section is to study the role of depth and negligibility in the structures $\E_s$ and $\E_w$. Let us start with an easy result.

\begin{theorem}\label{thm:mass_prolems_negligibility}
The collection of negligible $\Pi^0_1$ classes forms a filter in both $\E_s$ and $\E_w$.
\end{theorem}

\begin{proof}
Let $\C$ and $\D$ be negligible $\Pi^0_1$ classes. By Proposition \ref{prop:carac_neglible}, we have that $\lambda(\C^{\leq_T})=0$ and $\lambda(\D^{\leq_T})=0$. But since 
\[
\inf(\C,\D)^{\leq_T}=\C^{\leq_T}\cup\D^{\leq_T},
\]
it follows that $\lambda(\inf(\C,\D)^{\leq_T})=0$, which shows that $\inf(\C,\D)$ is negligible. Thus, the degrees of negligible classes in both $\E_w$ and $\E_s$ are closed under $\inf$.

%
%

Let $\mathcal{D}$ be non-negligible $\Pi^0_1$ and  $\mathcal{C}\leq_w\mathcal{D}$.  For each $i$, we define
\[
\D_i:=\{X\in\D:X\in\dom(\Phi_i)\;\&\;\Phi_i(X)\in\C\}.
\]
Since $\C\leq_w\D$, it follows that $\D=\bigcup_i\D_{i}$.  Furthermore, we have $\D^{\leq_T}=\bigl(\bigcup_i\D_{i}\bigr)^{\leq_T}=\bigcup_i\D_{i}^{\leq_T}$.  Since $\D$ is non-negligible, we have
\[
0<\lambda(\D^{\leq_T})\leq\sum_i\lambda(\D_{i}^{\leq_T}),
\]
and thus $\lambda(\D_{k}^{\leq_T})>0$ for some $k$.  But since $\Phi_{k}(\D_{k})\subseteq\C$, it follows that $\lambda(\C^{\leq_T})>0$, thus $\C$ is non-negligible.  Thus, negligibility is closed upwards under $\leq_w$ (and a fortiori, under $\leq_s$ as well).
\end{proof}

\begin{remark}
Simpson proved in~\cite{Simpson2006} that in $\E_w$, the complement of the filter of negligible classes is in fact a principal ideal, namely the ideal generated by the class $\inf(\sPA,2\mathcal{RAN})$, where $2\mathcal{RAN}$ is the class of $2$-random sequences, i.e. the sequences that are Martin-L\"of random relative to $\emptyset'$. 
\end{remark}

For the next two theorems, we need the following fact.

\begin{fact}\label{fact:tt-medvedev}
Let $\C$ and $\D$ be $\Pi^0_1$ classes such that $\C\leq_s\D$.  Then there is a total Turing functional $\Psi$ such that $\Psi(\D)\subseteq\C$.
\end{fact}

\begin{theorem}\label{thm:filter-medvedev}
The collection of deep $\Pi^0_1$ classes forms a filter in $\E_s$.
\end{theorem}

\begin{proof}
Let $\C$ and $\D$ be deep $\Pi^0_1$ classes with associated co-c.e.\ trees $S$ and~$T$, respectively.  Moreover, let $g$ and $h$ be computable order functions such that $\M(S_n)\leq 2^{-g(n)}$ and $\M(T_n) \leq 2^{-h(n)}$.  We define $f(n)=\min\{g(n),h(n)\}$, which is clearly a computable order function. It follows immediately that
$\M(S_n)\leq 2^{-f(n)}$ and $\M(T_n) \leq 2^{-f(n)}$.  

Now the co-c.e.\ tree associated with $\inf(\C,\D)$ is
\[
R=\{0^\frown\sigma:\sigma\in S\}\cup\{1^\frown\tau:\tau\in T\}.
\]
Consider the class~$\inf(\C,\D)$.  Setting $0^\frown S_n=\{0^\frown\sigma:\sigma\in S_n\}$ and $1^\frown T_n=\{1^\frown\sigma:\sigma\in T_n\}$ for each $n$, the co-c.e.\ tree associated to $\inf(\C,\D)$ is $0^\frown S_n\, \cup\, 1^\frown T_n$. Moreover,
\[
\M(0^\frown S_n \cup 1^\frown T_n) = \M(0^\frown S_n) + \M(1^\frown T_n) \leq^\times 2^{-f(n)} + 2^{-f(n)}
\]
which shows that $\inf(\C,\D)$ is deep. 

%

Next, suppose that $\mathcal{C}$ is a deep $\Pi^0_1$ class  and $\mathcal{D}$ is a $\Pi^0_1$ class satisfying  $\mathcal{C} \leq_s \mathcal{D}$ via the Turing functional $\Psi $, which we can assume to be total by Fact \ref{fact:tt-medvedev}. Let~$S$ and~$T$ be the co-c.e.\ trees associated to $\C$ and $\D$, respectively, and let $g$ be a computable order function such that $\M(S_n)\leq 2^{-g(n)}$.  Since $\Psi$ is total, its use is bounded by some computable function~$f$.  It follows that for every $\sigma\in T_{f(n)}$, there is some $\tau\in S_n$ such that $(\sigma,\tau)\in S_\Psi$ (the c.e.\ set of pairs of strings that generates $\Psi$).  Thus $\llb T_{f(n)}\rrb\subseteq\Psi^{-1}(\llb S_n\rrb)$.  Now let~$\Phi$ be a universal Turing functional such that $\M=\lambda_\Phi$.  Then:
\[
\M(T_{f(n)})=\lambda(\Phi^{-1}( T_{f(n)}))\leq\lambda(\Phi^{-1}(\llb T_{f(n)}\rrb)\leq\lambda(\Phi^{-1}(\Psi^{-1}(\llb S_n\rrb)))\leq\lambda_{\Psi\circ\Phi}(S_n),
\]
where the last inequality holds because $\Phi^{-1}(\Psi^{-1}(\llb S_n\rrb))=(\Psi\circ\Phi)^{-1}(\llb S_n\rrb)\subseteq\dom(\Psi\circ\Phi)$ and hence $(\Psi\circ\Phi)^{-1}(\llb S_n\rrb)\subseteq(\Psi\circ\Phi)^{-1}(S_n)$. Since $\lambda_{\Psi\circ\Phi}\leq^\times \M$, we have
\[
\M(T_{f(n)})\leq\lambda_{\Psi\circ\Phi}(S_n)\leq^\times \M(S_n)\leq 2^{-g(n)}.
\]
Thus, $\M(T_{f \circ g^{-1}(n)}) \leq 2^{-n}$, which shows that $\D$ is deep.


\end{proof}

The invariance of the notion of depth under Medvedev-equivalence is of importance for the next section. There, we prove that certain classes of objects, which are not necessarily infinite binary sequences, are deep. To do so, we fix a certain encoding of these objects by infinite binary sequences and prove the depth of the corresponding encoded class. By the above theorem, the particular choice of encoding is irrelevant: if the class is deep for an encoding, it will be deep for another encoding, as long as switching from the first encoding to the second one can be done computably and uniformly. 

One could ask whether we have a similar result in the lattice~$\E_w$. However, Theorem~\ref{thm:depth-not-Turing} shows that depth is not invariant under Muchnik equivalence: if we apply this theorem to a deep class~$\C$, we get a $\Pi^0_1$ class~$\D$ which is clearly Muchnik equivalent to~$\C$ but is not deep itself. Thus depth is Medvedev invariant but not Muchnik invariant.

More surprisingly, $\tt$-depth~\emph{is} a both Medvedev invariant and Muchnik invariant, as the $\tt$-deep classes form a filter in both lattices.

\begin{theorem}\label{thm:mass_problems_tt-deep}
The collection of $\tt$-deep $\Pi^0_1$ classes forms a filter in both $\E_s$ and $\E_w$. 
\end{theorem}

\begin{proof}

Suppose that $\C \leq_w \D$ and $\D$ is not $\tt$-deep.  By definition, this means that $\D$ has positive $\mu$-measure for some computable probability measure~$\mu$. Let $\{\mathcal{U}_k\}_{k \in \N}$ be the universal $\mu$-Martin-L\"of test and define $\mathcal{R}_k$ to be the complement of $\mathcal{U}_k$ (since $\U_k$ is $\Sigma^0_1$, $\mathcal{R}_k$ is $\Pi^0_1$). Since $\mu(\mathcal{R}_k) \geq 1- 2^{-k}$ for every $k$, there must be a~$j$ such that $\mu(\D \cap \mathcal{R}_j) > 0$ and in particular $\D \cap \mathcal{R}_j \not= \emptyset$. By the hyperimmune-free basis theorem, there is some $X\in\D \cap \mathcal{R}_j$ of hyperimmune-free Turing degree. Since $\C \leq_w \D$, $X$ must compute some element~$Y$ of~$\C$. But since~$X$ is of hyperimmune-free degree, $X$ in fact $\tt$-computes $Y$, i.e., $\Psi(X)=Y$ for some total Turing functional~$\Psi$. 

Since~$X$ is $\mu$-Martin-L\"of random, by the preservation of randomness theorem (see, for instance, Theorem 3.2 in \cite{BienvenuP2012}), $Y$ is Martin-L\"of random with respect to the computable measure $\mu_\Psi$ defined by $\mu_\Psi(\sigma):=\mu(\Psi^{-1}(\sigma))$ for every $\sigma$.  Thus, by Proposition~\ref{prop:carac-tt-negl}, $\C$ is not $\tt$-deep. This shows, a fortiori, that if $\C \leq_s \D$ and~$\D$ is not $\tt$-deep, then $\C$ is not $\tt$-deep. Thus $\tt$-depth is closed upwards in $\E_w$ and $\E_s$ and in particular is compatible with the equivalence relations $\equiv_s$ and $\equiv_w$. \\

Next, suppose that $\C$ and $\D$ are $\tt$-deep classes but that $\inf(\C,\D)$ is not $\tt$-deep (recall that the $\inf$ operator is the same for both $\E_s$ and $\E_w$).  Then by Theorem~\ref{prop:carac-tt-negl}, $\inf(\C,\D)$ contains a sequence $X$ that is $\mu$-Martin-L\"of random for some computable measure $\mu$.  Then for $i=0,1$ we define computable measures $\mu_i$ such that $\mu_i(\sigma)=\mu(i^\frown\sigma)/\mu(i)$ for every $\sigma$ (where this ratio is equal to~$0$ if $\mu(i)=0$).  It is routine to check that $Y$ is $\mu_i$-Martin-L\"of random if and only if $i^\frown Y$ is $\mu$-Martin-L\"of random for $i=0,1$.  Since $X=i^\frown Z$ for some $i=0,1$ and $Z\in\cs$, it follows that $Z$ is $\mu_i$-Martin-L\"of random.  But then $Z$ is contained in either $\C$ or $\D$, which contradicts our hypothesis that $\C$ and $\D$ are both $\tt$-deep.

\end{proof}


\section{Examples of deep $\Pi^0_1$ classes}\label{sec:examples}

In this section, we provide a number of examples of deep $\Pi^0_1$ classes that naturally occur in computability theory and algorithmic randomness. We give a uniform treatment of all these classes, i.e., we give a general method to prove the depth of $\Pi^0_1$ classes. 

\smallskip


\subsection{Consistent completions of Peano Arithmetic}


As mentioned in Section \ref{sec:neg}, Jockusch and Soare proved in \cite{JockuschS1972} that the $\Pi^0_1$ class $\sPA$ of consistent completions of $\PA$ is negligible.  However, as shown implicitly by Levin in \cite{Levin2013} and Stephan in \cite{Stephan2002}, $\sPA$ is also deep.   We will reproduce this result here.  

Following both Levin and Stephan, we will use fact that the class of consistent completions of $\PA$ is Medvedev equivalent to the class of total extensions of a universal, partial-computable $\{0,1\}$-valued function.  Thus, by showing the latter class is deep, we thereby establish that the former class is deep (via Theorem~\ref{thm:filter-medvedev}).
 
\begin{theorem}[Levin \cite{Levin2013}, Stephan \cite{Stephan2002}] \label{thm:levin-stephan}
Let $(\phi_e)_{e\in\N}$ be a standard enumeration of all $\{0,1\}$-valued partial computable functions.  Let~$u$ be a  function that is universal for this collection, e.g., defined by $u(\langle e,x\rangle) = \phi_e(x)$. Then the class~$\C$ of total extensions of~$u$ is a deep $\Pi^0_1$ class.  
\end{theorem}



\begin{proof}
We build a partial computable function $\phi_e$, whose index we know in advance by the recursion theorem. This means that we control the value of $u(\langle e,x\rangle)$ for all~$x$. First, we partition $\N$ into consecutive intervals $I_1, I_2, ...$ such that we control $2^{k+1}$ values of~$u$ inside~$I_k$. For each~$k$ in parallel, we define $\phi_e$ on $I_k$ as follows. 
\begin{itemize}
\item[] Step 1: Wait for a stage~$s$ such that the set $E_k[s]=\{\sigma \colon \sigma \uh I_k ~ \text{extends}~ u[s] \uh I_k \}$ is such that $\M(E_k)[s] \geq 2^{-k}$. 

\smallskip

\item[] Step 2: Find a $y \in I_k$ that we control and on which $u[s]$ is not defined. Consider the two ``halves"  $E^0_k[s]=\{ \sigma \in E_k[s] \colon \sigma(y)=0\}$ and $E^1_k[s]=\{\sigma \in E_k[s] \colon \sigma(y)=1\}$ of~$E_k[s]$. Note that either $\M(E^0_k[s])\geq 2^{-k-1}$ or $\M(E^1_k)[s] \geq 2^{-k-1}$. If the first holds, set $u(y)[s+1]=1$, otherwise set $u(y)[s+1]=0$. Go back to Step 1. 
\end{itemize}

The co-c.e.\ tree~$T$ associated to the class~$\C$ is the set of strings~$\sigma$ such that~$\sigma$ is an extension of $u \uh |\sigma|$. The construction works because every time we pass by Step 2, we remove from~$T[s]$ a set $E^i_k[s]$ (for some $i\in\{0,1\}$ and $k,s\in\N$) such that $\M(E^i_k)[s]\geq2^{-k-1}$. Therefore, Step 2 can be executed at most $2^{k+1}$ times, and by the definition of $I_k$ we do not run out of values~$y \in I_k$ on which we control~$u$. Therefore, the algorithm eventually reaches Step 1 and waits there forever. Setting $f(k)=\max(I_k)$, this implies that the $\M$-weight of the set $\{\sigma \colon \sigma \uh f(k) ~ \text{extends}~ u\}$ is bounded by $2^{-k}$, or equivalently,
\[
\M(T_{f(k)}) \leq 2^{-k},
\]
which proves that~$\C$ is deep. 
\end{proof}

The proof provided here gives us a general template to prove the depth of a $\Pi^0_1$ class. First of all, the definition of the $\Pi^0_1$ class should allow us to control parts of it in some way, either because we are defining the class ourselves or because, as in the above proof, the definition of the class involves some universal object which we can assume to partially control due to the recursion theorem. All the other examples of deep $\Pi^0_1$ class we will see below belong to this second category. Let us take a step back and analyze more closely the structure of the proof of Theorem~\ref{thm:levin-stephan}.  Given a $\Pi^0_1$ class $\C$ with canonical co-c.e.\ tree $T$, the proof consists of the following steps.\\ 

\noindent (1) For a given~$k$, we identify a level~$N=f(k)$ at which we wish to ensure $\M(T_N) \leq 2^{-k}$. The choice of~$N$ will depend on the particular class $\C$. \\

\noindent (2) Next we implement a two-step strategy to ensure that $\M(T_N)\leq2^{-k}$. Such a strategy will be called a \emph{$k$-strategy}. 

\begin{itemize}
\item[] (2.1) First, we wait for a stage~$s$ at which $\M(T_N)[s] \geq 2^{-k}$. 
\smallskip
\item[](2.2)  If at some stage $s$ this occurs, we remove one or several of the nodes of $T[s]$ at level~$N$ such that the total $\M[s]$-weight of these nodes is at least $\delta(k)$ for a certain function $\delta$, and then go back to Step 2. 
\end{itemize}

\noindent (3) Each execution of Step 2 (removing nodes from~$T[s]$ for some $s$) comes at a cost $\gamma(k)$ for some $k$, and we need to make sure that we do not go over some maximal total cost $\Gamma(k)$ throughout the execution of the $k$-strategy. \\

In the above example, $\delta(k)=2^{-k-1}$ and the only cost for us is to define one value of $u(y)$ out of the $2^{k+1}$ we control, so we can, for example, set $\gamma(k)=1$ and $\Gamma(k)=2^{k+1}$. By definition of~$\M$, the $k$-strategy can only go through Step 2 at most $1/\delta(k)$ times, and thus the total cost of the $k$-strategy will be at most $\gamma(k)/\delta(k)$. All we have to do is to make sure that $\gamma(k)/\delta(k) \leq \Gamma(k)$ (which is the case in the above example). Again, we have some flexibility on the choice of $N=f(k)$, so it will suffice to choose an appropriate~$N$ to ensure $\gamma(k)/\delta(k) \leq \Gamma(k)$. In some cases, there will not be a predefined maximum for each~$k$, but rather a global maximum $\Gamma$ that the sum of the the costs of the $k$-strategies should not exceed, i.e., we will want to have $\sum_k \gamma(k)/\delta(k) \leq \Gamma$. 
Let us now proceed to more examples of $\Pi^0_1$ classes. \\

\subsection{Shift-complex sequences}

\begin{definition}
Let $\alpha \in (0,1)$ and $c$ be a non-negative integer. 
\begin{itemize}
\item[(i)] $\sigma\in\str$ (resp. $X\in\cs$) is said to be \emph{$(\alpha,c)$-shift complex} if $\K(\tau) \geq \alpha |\tau| - c$ for every substring~$\tau$ of $\sigma$ (resp. of~$X$). 
\item[(ii)] $X\in\cs$ is said to be \emph{$\alpha$-shift complex} if it is $(\alpha,c)$-shift complex for some~$c$. 
\end{itemize}
\end{definition}

The very existence of $\alpha$-shift complex sequences is by no means obvious. Such sequences were first constructed by Durand, Levin, and Shen~\cite{DurandLS2008} who showed that there exist $\alpha$-shift complex sequences for all~$\alpha\in(0,1)$. It is easy to see that for every computable pair $(\alpha,c)$, the class of $(\alpha,c)$-shift complex sequences is a $\Pi^0_1$ class. Rumyantsev~\cite{Rumyantsev2011} proved that the class of $(\alpha,c)$-complex sequences is always negligible, but in fact, his proof essentially shows that it is even deep. The cornerstone of Rumyantsev's theorem is the following lemma. It relies on an ingenious combinatorial argument that we do not reproduce here. We refer the reader to~\cite{Rumyantsev2011} for the full proof. 

\begin{lemma}[Rumyantsev {\cite[Lemma 6]{Rumyantsev2011}}]\label{lem:rumyantsev}
Let $\beta \in (0,1)$.  For every rational $\eta\in(0,1)$, there exist two integers $n$ and $N$, with $n<N$, such that the following holds. For every probability distribution~$P$ on~$\{0,1\}^N$, there exist finite sets of strings $A_n, A_{n+1}, ..., A_N$ such that
\begin{itemize}
\item[(i)] for all~$j \in [n,N]$, $A_j$ contains only strings of length~$j$ and has at most $2^{j\beta }$ elements; and
\item[(ii)] $P(\{\sigma\in\{0,1\}^N\colon \sigma\text{ has some substring in }\cup_{i=n}^N A_i\})\geq 1-\eta$.
\end{itemize}
Moreover, $n$ and $N$ can be effectively computed from $\eta$ and can be chosen to be arbitrarily large. Once~$n$ and $N$ are fixed, the sets $A_i$ can be computed uniformly in~$P$. 
\end{lemma}

Rumyantsev does not explicitly state that the conclusion holds for \emph{all}  probability measures, but nothing in his proof makes use of a particular measure.  

\begin{theorem}\label{thm:shiftcomplex}
For any computable $\alpha\in(0,1)$ and integer $c \geq0$, the $\Pi^0_1$ class of $(\alpha,c)$-shift complex sequences is deep. 
\end{theorem}

\begin{proof}
Let $\C$ be the $\Pi^0_1$ class of $(\alpha,c)$-shift complex sequences and let~$T$ be its canonical co-c.e.\ tree. We shall build a left-c.e.\ discrete semi-measure~$m$ whose coding constant, which we will write in the form~$2^e$, we know in advance. This means that whenever we will set a string~$\sigma$ to be such that $m(\sigma) > 2^{-\alpha |\sigma| + c + e +1 }$, then automatically we will have $\m(\sigma) > 2^{-\alpha |\sigma| + c +1}$, and thus $\K(\sigma) < \alpha |\sigma| - c$. This will \textit{de facto} remove $\sigma$ from~$T$.

Let us now turn to the construction. First, we pick some $\beta$ such that $0 < \beta < \alpha$. For each~$k$, we apply the above Lemma \ref{lem:rumyantsev} to $\beta$ and $\eta = 1/2$ to obtain a pair $(n,N)$ with the above properties as described in the statement of the lemma (we will also make use of the fact that $n$ can be chosen arbitrarily large, see below). Then the two step strategy is the following:
\begin{itemize}
\item[] Step 1: Wait for a stage~$s$ at which $\M(T_N)[s] \geq 2^{-k}$. Up to delaying the increase of $\M$, we can assume that when such a stage occurs, one in fact has  $\M(T_N)[s] = 2^{-k}$.\footnote{That is, if $s$ satisfies $\M(T_N)[s-1] < 2^{-k}$ and $\M(T_N)[s]=r> 2^{-k}$, we can modify the enumeration of $\M$ so that $\M(T_N)[s]=2^{-k}$ by delaying the enumeration of the additional $\M$-measure of size $r-2^{-k}$ until stage $s+1$.}

\smallskip

\item[] Step 2: Let $P$ be the (computable) probability distribution on $\{0,1\}^N$ whose support is contained in $T_N[s]$ and such that for all~$\sigma \in T_N[s]$, $P(\sigma)=2^k \cdot \M(\sigma)[s]$. By Lemma~\ref{lem:rumyantsev}, we can compute a collection of finite sets $A_i$ such that:

\smallskip
\begin{itemize}
\item[(i)] For all~$j \in [n,N]$, $A_j$ contains only strings of length~$j$ and has at most $2^{j\beta}$ elements. 
\smallskip
\item[(ii)]  Setting $F=\{\sigma\in\{0,1\}^N\colon \sigma\text{ has some substring in }\cup_{i=n}^N A_i\}$, we have $P(F)\geq 1-\eta=1/2$.  Thus, $\M(T_N \cap F)[s] \geq 2^{-k-1}$. 
\end{itemize}
\smallskip
Then, for each~$i$ and each $\sigma \in A_i$, we ensure, increasing~$m$ if necessary, that  $m(\sigma) > 2^{-\alpha |\sigma| + c + e +1 }$. As explained above, this ensures that all strings in~$F$ are removed from~$T_N$ at that stage, and therefore we have removed a $\M$-weight of at least $\delta(k)=2^{-k-1}$ from~$T_N$. Then we go back to Step 1. 
\end{itemize}

To finish the proof, we need to make sure that this algorithm does not cost us too much. The constraint here is that we need to make sure that $\sum_{\tau\in\str} m(\tau) \leq 1$, so we are trying to stay under a global cost of $\Gamma=1$. For a given~$k$, our cost at each execution of Step 2 is the increase of~$m$ on the strings in~$F$. The total $m$-weight we add during one such execution of Step 2 is at most
\[
\gamma(k) = \sum_{j=n}^N |A_j| \cdot 2^{-j\alpha  + c + e +1 } \leq \sum_{j=n}^N 2^{j(\beta-\alpha) + c+ e + 1} \leq \frac{2^{(\beta-\alpha) n + c+ e + 1}}{1-2^{\beta-\alpha}}
\]
But since $n=n(k)$ can be chosen arbitrarily large, therefore, with an appropriate choice of $n$, we can make $\gamma(k) \leq 2^{-2k-2}$. Therefore, since the $k$-strategy executes Step 2 at most $1/\delta(k)$ times, we have
\[
\sum_{\tau\in\str} m(\tau)\leq\sum_{k\in\N} \gamma(k)/\delta(k) \leq \sum_{k\in\N} 2^{-2k-2}/2^{-k-1} \leq 1
\]
and therefore~$m$ is indeed a discrete semi-measure.

\end{proof}

\subsection{DNC$_q$ functions}

Let $(\phi_e)_{e\in\N}$ be a standard enumeration of partial computable functions from $\N$ to $\N$. Let $J$ be an effectively optimal partial computable function. By this we mean that there exists a total computable function~$h$, which we may assume to be one-to-one, such that $J(h(e,x))=\phi_e(x)$ and  $|h(e,x)| \leq |x|+c_e$ for some constant~$c_e$ that depends on~$e$ only (where the length of an integer is the length of its binary representation). For example, one could take:
\[
J(2^e(2x+1))=\phi_e(x)
\] 
The following notion was studied in \cite{GreenbergM2011}.

\begin{definition}
Let $q: \N \rightarrow \N$ be a computable order function. The set $\DNC_q$ is the set of total functions $f: \N \rightarrow \N$ such that for all~$n$:
\begin{itemize}
\item[(i)] $f(n) \not= J(n)$ (where this condition is trivially satisfied if $J(n)$ is undefined), and
\item[(ii)] $f(n) < q(n)$.
\end{itemize}
\end{definition}

We should note that this is a slight variation of the standard definition of diagonal non-computability (used in~\cite{GreenbergM2011}), which is usually formulated in terms of the condition 
\begin{itemize}
\item[(i$'$)] $f(n)\neq\phi_n(n)$ 
\end{itemize}
instead of the condition (i) given above.  While this makes no difference for diagonal non-computability alone, using condition (i') would make the class $\DNC_q$  highly dependent on the particular choice of enumeration $(\phi_e)_{e\in\N}$, while with our definition, this dependence is somewhat reduced. More precisely, we have the following robustness property. Let $J_0$ and $J_1$ be two optimal partial computable functions, and for an arbitrary computable order function $q$, let $\DNC_{q}^0$ and $\DNC_{q}^1$ be the corresponding classes of functions given by the above definition. Then there are constants~$c,d$ such that, setting $\tilde{q}(n)=q(cn+d)$, we have 
\[
\DNC_{\tilde{q}}^0 \leq_s \DNC_{q}^1.
\]
Indeed there is a constant~$e$ such that $J_1(h(e,x))=J_0(x)$, and $h$ is bounded by a linear function.


Note that for every fixed computable order function~$q$, the class $\DNC_q$ is a $\Pi^0_1$ subset of $\N^\N$, but due to the bound~$q$, $\DNC_q$ can be viewed as a $\Pi^0_1$ subset of $\cs$. Whether the class $\DNC_q$ is deep turns out to depend on~$q$. Indeed, we have the following interesting dichotomy theorem. 

\begin{theorem}\label{thm:dnc-q}
Let $q$ be a computable order function. 
\begin{itemize}
\item[(i)] If $\sum_{n\in\N} 1/q(n) < \infty$, then $\DNC_q$ is not $\tt$-negligible. 
\item[(ii)] If  $\sum_{n\in\N} 1/q(n) = \infty$, then $\DNC_q$ is deep.
\end{itemize}
\end{theorem}

\begin{proof}
Part~(i) is a straightforward adaptation of Ku\v cera's proof that every Martin-L\"of random sequence computes a DNC function. Pick an~$X \in \cs$ at random and use it as a source  to randomly pick the value of $f(i)$ uniformly among $\{0,...,q(i)-1\}$, independently of the other values of $f$.


The details are as follows. First, we can assume that $q(n)$ is a power of~$2$ for all~$n$. Indeed, for $q' \leq q$ the class $\DNC_{q'}$ is contained in $\DNC_{q}$, so if we take $q'(n)$ to be the largest power of~$2$ less than or equal to $q(n)$, we have that $q'$ is computable, $q' \leq q$,  and $\sum_{n\in\N} 1/q'(n) < \infty$ because $q'$ is equal to~$q$ up to factor~$2$. Now, set $q(n)=2^{r(n)}$. Split $\N$ into intervals where interval $I_n$ has length~$r(n)$. One can now interpret any infinite binary sequence~$X$ as a function $f_X: \N \rightarrow \N$, where $X(n)$ is the index of $X \uh I_n$ in the lexicographic ordering of strings of length~$r(n)$. For~$X$ taken at random with respect to the uniform measure, the event $f_X(n) = J(n)$ has probability at most $1/q(n)$ and is independent of all such events for $n' \not= n$. Thus the total probability over~$X$ such that $f_X(n)\neq J(n)$ for all~$n$ is at least 
\[
\prod_{n\in\N} (1-1/q(n)),
\]
an expression that is positive if and only if $\sum_{n\in\N} 1/q(n)<  \infty$, which is satisfied by hypothesis (the proof of Lemma 5.6.4 in \cite{Nies2009} includes a proof of this result). Thus, the class $\DNC_q$, encoded as above, has positive uniform measure, and thus is not $\tt$-negligible.

For~(ii), let~$q$ be a computable order function such that $\sum_{n\in\N} 1/q(n) = \infty$. Let $T$ be the canonical co-c.e.\ binary tree in which the elements of $\DNC_q$ are encoded. By the recursion theorem, we will build a partial recursive function $\phi_e$ whose index~$e$ we know in advance and therefore will be able to define $J$ on the set of values $D = \{h(e,x) : x \in \N\}$, where $h$ is defined together with~$J$ at the beginning of this section. By definition of~$h$ and the assumption of injectivity, the set~$D$ has positive (lower) density in~$\N$. Thus there is a constant~$d$ such that for every interval~$I$ of type $\{2^{(k-1)d},...,2^{kd}-1\}$ with~$k$ large enough (say, greater than~$d$), $|D \cap I| \geq 2^{-d} |I|$. This in particular implies that
\[
\sum_{n \in D} 1/q(n) = \infty.
\]
To show this, we appeal to Cauchy's condensation test, according to which for any positive non-increasing sequence $(a_n)_{n\in\N}$, $\sum_{n\in\N} a_n < \infty$ if and only if $\sum_{k\in\N} 2^k a_{2^k} < \infty$. A trivial adaptation of the proof of this criterion gives that $\sum_{n\in\N} a_n < \infty$ if and only if $\sum_{k\in\N} m^k a_{m^k} < \infty$ for any integer~$m \geq 2$. Taking $m=2^d$, since the sum $\sum_{n\in\N}1/q(n)$ diverges, we have that $\sum_{k\in\N} 2^{kd}/q(2^{kd})$ diverges, and thus
\[
 \sum_{n \in D} \frac{1}{q(n)} \geq \sum_{k>d} \frac{|D \cap \{2^{(k-1)d},...,2^{kd}-1\}|}{q(2^{kd})} \geq \sum_{k>d} \frac{2^{-d} \cdot 2^{kd-1}}{q(2^{kd})} = \infty.
\]

Now that we have established that $\sum_{n \in D} 1/q(n) = \infty$, we remark that this is equivalent to having $\prod_{n \in D} (1-1/q(n)) = 0$. Thus, we can effectively partition $D$ into countably many finite sets $D_j$ such that $\prod_{n \in D_j} (1-1/q(n)) < 1/2$. We are ready to describe the construction. For each~$k$, reserve some finite collection $D_{j_1}, D_{j_2}, ..., D_{j_{k+1}}$ of sets $D_j$. Let $N$ be a level of the binary tree~$T$ sufficiently large such that the encoding of each path~$f$ up to length~$N$ is enough to recover the values of $f$ on $D_{j_1} \cup D_{j_2} \cup  ...\cup D_{j_{k+1}}$, which can be found effectively in~$k$. The $k$-strategy then works as follows.

\begin{itemize}
\item[] Initialization. Set $i=1$.  

\smallskip

\item[] Step 1: Wait for a stage~$s$ at which $\M(T_N)[s] \geq 2^{-k}$. Up to delaying the increase of $\M$, we can assume that when such a stage occurs, one in fact has  $\M(T_N)[s] = 2^{-k}$.

\smallskip

\item[] Step 2: Since each~$\sigma \in T_N[s]$ can be viewed as a function from some initial segment of $\N$ that contains $D_{j_1} \cup D_{j_2} \cup  ...\cup D_{j_{k+1}}$, take the first value $x \in D_{j_i}$. For all $\sigma \in T_N$, $\sigma(x) < q(x)$, thus by the pigeonhole principle there must be at least one value~$v < q(x)$ such that
\[
\M \big( \{ \sigma \in T_N : \sigma(x)=v\})[s] \geq 2^{-k}/q(x).
\]
Set $J(x)$ to be the least such value~$v$. This thus gives 
\[
\M \big( \{ \sigma \in T_N : \sigma(x) \not= v\})[s] \leq 2^{-k}(1-1/q(x)).
\]
Now take another $x'$ in $D_{j_i}$ on which~$J$ has not been defined yet. By the same reasoning, there must be a value $v' < q(x')$ such that 
\[
\M \big( \{ \sigma \in T_N : \sigma(x) \not= v \wedge \sigma(x') \not= v' \})[s] \leq 2^{-k}(1-1/q(x))(1-1/q(x')).
\]
We then set $J(x')$ to be equal to $v'$. Continuing in this fashion, we can assign all the values of $J$ on $D_{j_i}$ in such a way that
\[
\M \big( \{ \sigma \in T_N : \forall x \in D_{j_i}, \ \sigma(x) \not=J(x) \}\big)[s] \leq 2^{-k} \prod_{n \in D_j} (1-1/q(n)) \leq 2^{-k}/2.
\]
Then we increment~$i$ by~$1$ and go back to Step 1. 
\end{itemize}
Once again, at each execution of Step 2 we remove an $\M$-weight of at least $\delta(k)=2^{-k-1}$ from~$T_N$, since setting $\sigma(x)=J(x)$ for some~$x$ immediately ensures that no extension of~$\sigma$ is in $\DNC_q$. Moreover, each execution of Step 2 requires us to define $J$ on all values in some $D_{j_i}$ for $1\leq j\leq k+1$.  That is, each such execution costs us one $D_{j_i}$, so as in the case of consistent completions of $\PA$, we can set $\gamma(k)=1$ and $\Gamma(k)=2^{k+1}$, which ensures that $\gamma(k)/\delta(k)\leq\Gamma(k)$. Therefore, $\DNC_q$ is a deep $\Pi^0_1$ class.


\end{proof}


%

\subsection{Finite sets of maximally complex strings}

For a constant~$c>0$, generating a long string~$\sigma$ of high Kolmogorov complexity, for example $\K(\sigma) > |\sigma|-c$, can be easily achieved with high probability if one has access to a random source (just repeatedly flip a fair coin and output the raw result). One can even use this technique to generate a sequence of strings $\sigma_1, \sigma_2, \sigma_3, \ldots , $ such that $|\sigma_n| = n$ and $\K(\sigma_n) \geq n-c$. Indeed, the measure of the sequences~$X$ such that $\K(X \uh n) \geq n -d$ for all~$n$ is at least $1-2^{-d}$, and thus the $\Pi^0_1$ class of such sequences of strings $(\sigma_1,\sigma_2,...,)$ of high complexity (encoded as elements of $\cs$) is not even $\tt$-negligible. \\

The situation changes dramatically if one wishes to obtain many distinct high complexity strings of a given length. Let $\ell,f,d:\N\rightarrow\N$ be any computable functions. Consider the $\Pi^0_1$ class $\KK_{f,\ell,d}$ whose members are sequences $\vec F=(F_1, F_2, F_3, ...)$ where for all~$i$, $F_i$ is a finite set of~$f(i)$ strings~$\sigma$ of length~$\ell(i)$ such that $\K(\sigma) \geq \ell(i)-d(i)$. Note once again that $\KK_{f,\ell,d}$ can be viewed, modulo encoding, as a $\Pi^0_1$ subclass of $\cs$. 

\begin{theorem}\label{thm:many-complex-strings}
For all computable functions~$f,\ell,d$ such that $f(i)/2^{d(i)}$ takes arbitrarily large values and $\ell$ is increasing, the class $\KK_{f,\ell,d}$ is deep. 
\end{theorem}

\begin{proof}
Let~$T$ be the canonical co-c.e.\ tree associated to the class $\KK_{f,\ell,d}$. Just like in the proof of Theorem~\ref{thm:shiftcomplex}, we will build a discrete semi-measure~$m$ whose coding constant~$2^e$ we know in advance and thus, by setting $m(\sigma) > 2^{-|\sigma| + d(|\sigma|) + e +1 }$, we will force $\m(\sigma) > 2^{-|\sigma| + d(|\sigma|) +1}$, and thus $\K(\sigma) < |\sigma| - d(|\sigma|)$, which will consequently remove from $T$ every sequence of sets $(F_i)_{i\in\N}$ such that~$\sigma$ belongs to some $F_i$. 

Fix a~$k$, and let us pick some well-chosen $i=i(k)$, to be determined later. For readability, let $f=f(i(k))$, $\ell=\ell(i(k))$ and $d=d(i(k))$. Let~$N$ be the level of~$T$ at which the first~$i$ sets $F_1\dotsc F_i$ of the sequence $\vec{F}$ of~$\KK_{f,\ell,d}$ are encoded. The $k$-strategy does the following.

\begin{itemize}
\item[] Step 1: Wait for a stage~$s$ at which $\M(T_N)[s] \geq 2^{-k}$. Up to delaying the increase of $\M$, we can assume that when such a stage occurs, one in fact has  $\M(T_N)[s] = 2^{-k}$.

\smallskip

\item[] Step 2: The members of $T_N[s]$ consist of finite sequences $F_1, ..., F_i$ of sets of strings where $F_i$ contains~$f$ distinct strings of length~$\ell$. By the pigeonhole principle, there must be a string $\sigma$ of length~$\ell$ such that 
\[
\M \bigl( \{ (F_1,\dotsc,F_i) \in T_N \colon \sigma \in F_i \}\bigr)[s] \geq 2^{-k} \cdot f \cdot 2^{-\ell}
\]
Indeed, there is a total $\M$-weight $2^{-k}$ of possible sets $F_i$, each of which contains~$f$ strings of length~$\ell$, thus $\sum_{|\sigma|=\ell} \M \bigl( \{ (F_1,\dotsc,F_i) \in T_N \colon \sigma \in F_i \}\bigr)[s] \geq 2^{-k} \cdot f$, and there are $2^\ell$ strings of length~$\ell$ in total. 
Effectively find such a string~$\sigma$, and set $m(\sigma) > 2^{-\ell + d + e +1 }$. By the choice of $\sigma$, this causes an $\M$-weight of at least $2^{-k} \cdot f \cdot 2^{-\ell}$ of nodes of $T_N[s]$ to leave the tree. Then go back to Step 1. 
\end{itemize}

As in the previous examples, it remains to conduct the ``cost analysis". The resource here again is the weight we are allowed to assign to $m$, which has to be bounded by $\Gamma=1$. At each execution of Step 2 of the $k$-strategy, our cost is the increase of $m$, which is at most of $\gamma(k)=2^{-\ell + d + e +1 }$, while an $\M$-weight of at least $\delta(k) = 2^{-k} \cdot f \cdot 2^{-\ell}$ leaves the tree~$T_N$. This gives a ratio $\gamma(k)/\delta(k) = 2^{d+e+k+1}/f$, which bounds the total cost of the $k$-strategy, so we want $\sum_k \gamma(k)/\delta(k) \leq 1$, i.e., we want:
\[
\sum_k 2^{d(i(k)) + e +k+1}/f(i(k))\leq 1
\]
By assumption, $f(i)/2^{d(i)}$ takes arbitrarily large values. It thus suffices to choose $i(k)$ such that $2^{d(i(k)) + e+1}/f(i(k)) \leq 2^{-2k-1}$.


\end{proof}

\subsection{Compression functions}

Our next example of a family of deep classes is given in terms of compression functions, which were introduced by Nies, Stephan, and Terwijn in \cite{NiesST2005} to provide a characterization of
2-randomness (that is, $\emptyset'$-Martin-L\"of randomness) in terms of incompressibility.

\begin{definition}
A $\K$-compression function with constant~$c>0$ is a function $g: \str \rightarrow \N$ such that $g \leq \K+c$ and $\sum_{\sigma\in\str} 2^{-g(\sigma)} \leq 1$.  We denote by $\CF_c$ be the class of compression functions with constant $c$. 
\end{definition}

The condition $g(\sigma) \leq \K(\sigma)+c$ implies that $g(\sigma) \leq 2|\sigma| + c +c' $ for some fixed constant~$c'$, and therefore $\CF_c$ can be seen as a $\Pi^0_1$ subclass of $\cs$, modulo encoding the functions as binary sequences. 
Of course $\CF_c$ contains the function $\K$ itself so it is a non-empty class. We now show the following. 

\begin{theorem}
For all~$c\geq 0$, the class $\CF_c$ is deep. 
\end{theorem}

\begin{proof}
Although we could give a direct proof following the same template as the previous examples, we will instead show that for all~$c$, we have $\KK_{f,\ell,d} \leq_s \CF_c$ for some computable~$f, \ell, d$ (where $\KK_{f,\ell,d}$ is the class we defined in the previous section) such that $\ell$ is increasing and $f/2^d$ takes arbitrarily large values, which by Theorem~\ref{thm:filter-medvedev} implies the depth of~$\CF_c$. 

Let $g \in \CF_c$. For all~$n$, since $\sum_{|\sigma|=n} 2^{-g(\sigma)} \leq 1$, there are at most $2^{n-1}$ strings of length~$n$ such that $g(\sigma) < n - 1$, and thus at least $2^{n-1}$ strings $\sigma$ of length~$n$ such that $g(\sigma) \geq n - 1$.  Since $g \in \CF_c$, for each~$\sigma$ such that $g(\sigma) \geq |\sigma| - 1$, we also have $\K(\sigma) \geq |\sigma| - c - 1$. Thus, given~$g \in \CF_c$ as an oracle we can find, for each~$n \geq 3$, $2^{n-1}$ strings $\sigma$ of length~$n$ such that $\K(\sigma) \geq |\sigma| - c - 1$. Setting $f(i)=2^{i+2}$, $\ell(i)=i+3$ and $d(i)=c+1$, we have uniformly reduced $\KK_{f,\ell,d}$ to $\CF_{c}$. Since $d$ is a constant function it is obvious that $f(i)/2^d$ is unbounded, thus $\KK_{f,\ell,d}$ is deep (by Theorem~\ref{thm:many-complex-strings}) and by Theorem \ref{thm:filter-medvedev} so is~$\CF_c$.

\end{proof}

The above result is not tight: the proof in fact shows that if $d$ is not a constant function but is such that $2^n/d(n)$ takes arbitrarily large values, then the class of functions~$g$ such that $g(\sigma) \leq \K(\sigma) + d(|\sigma|)$ for all~$\sigma$ is a deep class. \\

\subsection{A notion related to $\High(\CR,\MLR)$}

In Section \ref{sec:lowness}, we discuss lowness for randomness notions. Here we look at a dual notion, highness for randomness, specifically the class $\High(\CR,\MLR)$, whose precise characterization is an outstanding open question in algorithmic randomness. We shall see that this class is tightly connected to the notion of depth. 

\begin{definition}
A sequence $A\in\cs$ is in the class $\High(\CR,\MLR)$ if  $\MLR \supseteq \CR^A$.
\end{definition}

The class $\High(\CR,\MLR)$ itself is not $\Pi^0_1$ (as it is closed under finite change of prefixes), but all its members have a ``deep property". Bienvenu and Miller~\cite{BienvenuM2012} proved that when $A$ is in $\High(\CR,\MLR)$, then for every c.e.\ set of strings $S$ such that $\sum_{\sigma \in S} 2^{-|\sigma|} < 1$, $A$ computes a martingale~$d$ such that $d(\emptystring)=1$ and for some fixed rational $r>0$, $d(\sigma)>1+r$ for all $\sigma \in S$. It is straightforward to show that real-valued martingales can be approximated by dyadic-valued martingales with arbitrary precision; in particular one can assume that $d(\sigma)$ is dyadic for all~$\sigma$ (and thus can be coded using $f(|\sigma|)$ bits for some computable function~$f$), still keeping the property $\sigma \in S \Rightarrow d(\sigma)>1+r$. For a well-chosen~$S$, such martingales form a deep class. 

\begin{theorem}
Let $S$ be the set of strings~$\sigma$ such that $\K(\sigma) < |\sigma|$ (so that we have $\sum_{\sigma \in S} 2^{-|\sigma|} < 1$). Let $g$ be a computable function and $r>0$ be a rational number. Define the class $\W_{g,r}$ to be the set of dyadic-valued martingales~$d$ such that $d(\emptystring)=1$, $d(\sigma)>1+r$ for all $\sigma \in S$ and such that for all~$\sigma$, $d(\sigma)$ can be coded using $g(|\sigma|)$ bits. Then $\W_{g,r}$ can be viewed as a $\Pi^0_1$ class of $\cs$, and this class is deep. 
\end{theorem}

\begin{proof}
Again, we are going to show the depth of $\W_{g,r}$ by  a Medvedev reduction from $\KK_{f,\ell,h}$ for some increasing computable function~$\ell$ and computable functions $f$ and $h$ such that $f/2^h$ is unbounded. More precisely, we will take $f(n)=n$, $\ell(n)=n + c$ (for a well-chosen $c$ given below) and $h(n)=0$. Now suppose we are given a martingale~$d \in \W_{f,r}$. For any given length~$n$, we have, by the martingale fairness condition,
\[
\sum_{\sigma: |\sigma|=n} d(\sigma) = 2^n.
\]
It follows that there are at most $2^n/(1+r)$ strings~$\sigma$ of length~$n$ such that $d(\sigma)>1+r$, and therefore at least $2^n-2^n/(1+r)$ strings such that $d(\sigma) \leq 1+r$. Having oracle access to~$d$, such strings can be effectively found and listed. Since $2^n-2^n/(1+r) > n -c$ for all~$n$ and some fixed constant~$c$, for each $n$ we can use $d$ to list $n$ strings $\sigma_1,\dotsc,\sigma_n$ of length~$n+c$ such that $d(\sigma_i) \leq 1+r$ for $1\leq i\leq n$. But by definition of $d$, $d(\sigma)\leq 1+r$ implies that $\sigma \notin S$, which further implies that $\K(\sigma) \geq |\sigma|$. This shows that $\W_{g,r}$ is above $\KK_{f,\ell,d}$ in the Medvedev degrees. By Theorem~\ref{thm:many-complex-strings},  $\KK_{f,\ell,h}$ is deep, and hence by Theorem \ref{thm:filter-medvedev}, $\W_{g,r}$ is deep as well. 

\end{proof}

The examples of deep classes provided in this section, combined with Theorem~\ref{thm:diff-deep}, give us the results mentioned page~\pageref{page:discussion-diff-deep}: If $X$ is a difference random sequence,\begin{itemize}
\item it does not compute any shift-complex sequence (Khan);
\item it does not compute any $\DNC_q$ function when $q$ is a computable order function such that $\sum_n 1/q(n) = \infty$ (Miller);
\item it does not compute any compression function (Greenberg, Miller, Nies);
\item it is not in $\High(\CR,\MLR)$ (Greenberg, Miller, Nies).
\end{itemize} 
The fact that a difference random cannot compute large sets of complex strings (in the sense of Theorem~\ref{thm:many-complex-strings}) appears to be new.

\section{Lowness and depth}\label{sec:lowness}

Various lowness notions have been well-studied in algorithmic randomness, where a lowness notion is given by a collection of sequences that are in some sense computationally weak.  Many lowness notions take the following form:  For a relativizable collection $\sS\subseteq\cs$, we say that $A$ is \emph{low for $\sS$} if $\sS\subseteq\sS^A$.  For instance, if we let $\sS=\MLR$, then the resulting lowness notion consists of the sequences that are \emph{low for Martin-L\"of random}, a collection we write as $\Low(\MLR)$.

Lowness notions need not be given in terms of relativizable classes.  For instance, if we let $\m^A$ be a universal $A$-left-c.e.\ discrete semi-measure, we can defined $A$ to be \emph{low for $\m$} if $\m^A(\sigma)\leq^\times\m(\sigma)$ for every $\sigma$.
In addition, some lowness notions are not given in terms of relativization, such as the notion of \emph{$\K$-triviality}, where $A\in\cs$ is $\K$-trivial if and only if $\K(A\uh n)\leq \K(n)+O(1)$ (where we take $n$ to be $1^n$).  Surprisingly, we have the following result (see~\cite{Nies2009} for a detailed survey of results in this direction).
\begin{theorem} Let $A\in\cs$.  The following are equivalent:
\begin{itemize}
\item[(i)] $A\in\Low(\MLR)$; 
\item[(ii)] $A$ is low for $\m$;
\item[(iii)] $A$ is $\K$-trivial.
\end{itemize}
\end{theorem}

The cupping problem, a longstanding open problem in algorithmic randomness involving $\K$-triviality, is to determine whether there exists a $\K$-trivial sequence $A$ and some Martin-L\"of random sequence $X\not\geq_T\emptyset'$ such that $X\oplus A\geq_T\emptyset'$.  A negative answer was recently provided by Day and Miller \cite{DayM-Sub}:

\begin{theorem}[Day-Miller~\cite{DayM-Sub}] \label{thm:day-miller}
A sequence~$A$ is K-trivial if and only if for every difference random sequence~$X$, $X \oplus A \not\geq_T \emptyset'$. 
\end{theorem}



Using the notion of depth, we can strengthen the Day-Miller result.  

\begin{theorem}\label{thm:k-trivial-deep-cupping}
Let $X$ be an incomplete Martin-L\"of random sequence and $A$  be $K$-trivial. Then $X \oplus A$ does not compute any member of any deep $\Pi^0_1$ class. 
\end{theorem}

\begin{proof}
First, we need to partially relativize the notion of depth:  for $A\in\cs$, a $\Pi^0_1$ class $\C\subseteq\cs$ is \emph{deep relative to $A$} if there is some computable order function $f$ such that $\m^A(T_{f(n)})\leq 2^{-n}$ if and only if there is some $A$-computable order function $g$ such that $\M^A(T_{g(n)})\leq 2^{-n}$ (where $\M^A$ is a universal $A$-left-c.e.\ continuous semi-measure).

Let $\C$ be a deep $\Pi^0_1$ class with canonical co-c.e.\ tree~$T$ and let~$f$ be a computable function such that $\m(T_{f(n)}) \leq 2^{-n}$. Since~$A$ is low for~$\m$, we have also have $\m^A(T_{f(n)}) \leq^\times 2^{-n}$, and thus $\C$ is deep relative to~$A$. 

Let~$X$ be a sequence such that $X \oplus A$ computes a member of $\C$ via a Turing functional $\Phi$ and suppose, for the sake of contradiction, that $X$ is difference random. Let
\[
\D_n = \{Z \, : \, \Phi^{Z \oplus A} \uh f(n)  \downarrow  \, \in T_{f(n)}\}.
\]
The set $\D_n$ can be written as the difference $\U_n \setminus \V_n$ of two $A$-effectively open sets (uniformly in~$n$) with $\U_n = \{Z \, : \, \Phi^{Z \oplus A} \uh f(n)  \downarrow\}$ and $\V_n= \{Z \, : \, \Phi^{Z \oplus A} \uh f(n)  \downarrow  \notin T_{f(n)}\}$

We can see the functional $Z \mapsto \Phi^{Z \oplus A}$ as an $A$-Turing functional $\Psi$, and thus by the univerality of $\M^A$ for the class of $A$-left-c.e.\ continuous semi-measures, we have $\M^A \geq^\times \lambda_\Psi$. By definition of $\D_n$, we therefore obtain:
\[
\lambda(\D_n) \leq \lambda_\Psi(T_{f(n)}) \leq^\times \M^A(T_{f(n)}) \leq 2^{-n}.
\]
This shows that the sequence~$X$, which by assumption belongs to all~$\D_n$, is not $A$-difference random. It is, however, $A$-Martin-L\"of random as~$A$ is low for Martin-L\"of randomness. Relativizing Theorem~\ref{thm:franklin-ng} to $A$, this shows that $X \oplus A \geq_T A'$. But this contradicts the Day-Miller theorem (Theorem~\ref{thm:day-miller}). 

\end{proof}

As we cannot compute any member of a deep class by joining a Martin-L\"of random sequence with a low for Martin-L\"of random sequence, it is not unreasonable to ask if there is a notion of randomness $\mathcal{R}$ such that we cannot $\tt$-compute any members of a $\tt$-deep class by joining an $\mathcal{R}$-random sequence with a low-for-$\mathcal{R}$ sequence.  We obtain a partial answer to this question using Kurtz randomness.

From the discussion of lowness at the beginning of this section, we have $A\in\Low(\KR)$ if and only if $\KR\subseteq\KR^A$.  Moreover, we define the class $\Low(\MLR,\KR)$ to be the collection of sequences $A$ such that $\MLR\subseteq\KR^A$.  Since $\MLR\subseteq\KR$, it follows that $\Low(\KR)\subseteq\Low(\MLR,\KR)$.  Miller and Greenberg \cite{GreenbergM2009} obtained the following characterization of $\Low(\KR)$ and $\Low(\MLR,\KR)$.  Recall that $A\in\cs$ is computably dominated if every $f\leq_T A$ is dominated by some computable function.

\begin{theorem}[Greenberg-Miller~\cite{GreenbergM2009}] \label{thm:greenberg-miller-kr}
Let $A\in\cs$.
\begin{itemize}
\item[(i)] $A\in\mathrm{Low}(\MLR,\KR)$ if and only if $A$ is of non-DNC degree. 
\item[(ii)] $A\in\mathrm{Low}(\KR)$ if and only if $A\in\mathrm{Low}(\MLR,\KR)$ and computably dominated. 
\end{itemize}
\end{theorem}

For $A\in\cs$, a $\Pi^0_1$ class $\C$ is \emph{$\tt$-negligible relative to $A$} if $\mu^A(\C)=0$ for every $A$-computable measure $\mu^A$.  We first prove the following.

\begin{proposition}\label{prop:low-for-tt-depth}
If $A\in\mathrm{Low}(\MLR,\KR)$, then every deep $\Pi^0_1$ class is $\tt$-negligible relative to~$A$. 
\end{proposition}

\begin{proof}
Let $\C$ be a deep $\Pi^0_1$ class and $T$ its associated canonical co-c.e.\ tree. Let $A\in\Low(\MLR,\KR)$, which by Theorem \ref{thm:greenberg-miller-kr} (i) is equivalent to being of non-DNC degree. We appeal to a useful characterization of non-DNC degrees due to H\"olzl and Merkle~\cite{HolzlM2010}: $A$ is of non-DNC degree if and only if it is \emph{infinitely often c.e.\ traceable} (hereafter, i.o.\ c.e.\ traceable). This means that there exists a computable order function~$h$ such that the following holds: for every total $A$-computable function~$s: \N \rightarrow \N$, there exists a family $(S_n)_{n\in\N}$ of uniformly c.e.\ finite sets such that $|S_n|<h(n)$ for all~$n$ and $s(n) \in S_n$ for infinitely many~$n$. 

Let~$h$ be a computable order function witnessing the i.o.\ c.e.\ traceability of $A$.  Since $\C$ is deep, let~$f$ be a computable function such that $\M(T_{f(n)}) \leq 2^{-2n}/h(n)$. 
Suppose for the sake of contradiction that $\C$ is not $\tt$-negligible relative to~$A$, which means that there exists an $A$-computable measure~$\mu^A$ such that $\mu^A(\C)>r$ for some rational~$r>0$. Let $s\leq_TA$ be the function that on input~$n$ gives a rational lower-approximation, with precision $1/2$, of the values of $\mu^A$ on all strings of length~$f(n)$ (encoded as an integer). By this we mean that $s(n)$ gives us for all strings $\sigma$ of length~$f(n)$ a rational value $s(n,\sigma)$ such that $\mu^A(\sigma)/2 \leq s(n,\sigma) \leq \mu^A(\sigma)$. Let $(S_n)_{n\in\N}$  witness the traceability of $s$, i.e., the $S_n$'s are uniformly c.e., $|S_n|<h(n)$ for every~$n$, and $s(n) \in S_n$ for infinitely many~$n$. 

We now build a left-c.e.\ continuous semi-measure~$\rho$ as follows. For all~$n$, enumerate $S_n$. For each member $z \in S_n$, interpret $z$ as a mass distribution $\nu$ on the collection of strings of length $f(n)$. Then, for each string~$\sigma$ of length~$f(n)$, increase $\rho(\sigma)$ (as well as strings comparable with $\sigma$ in any way that ensures that $\rho$ remains a semi-measure) by $2^{-n-1}\nu(\sigma)/h(n)$. This has a total cost of $2^{-n-1}/h(n)$, and since there are at most $h(n)$ elements in~$S_n$, the total cost at level~$f(n)$ is at most $2^{-n-1}$.  Therefore the total cost of the construction of~$\rho$ is bounded by~$1$ and thus~$\rho$ is indeed a left-c.e.\ continuous semi-measure.

Now, for any~$n$ such that $s(n) \in S_n$ (as there are infinitely many such~$n$),  $\rho$ distributes an amount of at least $(\mu^A(T_{f(n)})/2) \cdot (2^{-n-1}/h(n))$ on $T_{f(n)}$, and since $\mu^A(T_n)>r$, this gives
\[
\rho(T_{f(n)}) \geq 2^{-n-O(1)}/h(n).
\]
However, we had assumed that 
\[
\M(T_{f(n)}) \leq 2^{-2n}/h(n),
\]
for all~$n$.  But since $\rho \leq^\times \M$, we get a contradiction.

\end{proof}

The following result involves the notion of relative $\tt$-reducibility.  For a fixed $A\in\cs$, a $\tt(A)$-functional is a total $A$-computable Turing functional.  Equivalently, we can define a $\tt(A)$-functional $\Psi^A$ in terms of a Turing functional $\Phi$ as follows:  Let $\Phi$ be defined on all inputs of the form $X\oplus A$.  Then we set $\Psi^A(X)=\Phi(X\oplus A)$.  Furthermore, one can show that there is an $A$-computable bound on the use of $X$ in the computation (just as there is a computable bound in the use function for unrelativized $\tt$-computations).

\begin{theorem}\label{thm:kurtz-cupping}
Let $X$ be Kurtz random and $A\in\Low(\KR)$. Then $X$ does not $\tt(A)$-compute any member of any deep $\Pi^0_1$ class.\end{theorem}

\begin{proof}
Let $\C$ be a deep $\Pi^0_1$ class and $A\in\Low(\KR)$. By Proposition \ref{prop:low-for-tt-depth}, $\C$ is also $\tt$-deep relative to $A$. Let $\Phi^A$ be a $\tt(A)$-functional. The pre-image $\D$ of $\C$ under $\Phi^A$ is a $\Pi^0_1(A)$ class, which must be $\tt$-deep relative to $A$ as well, by Theorem~\ref{thm:mass_problems_tt-deep} relativized to~$A$. Now, applying Proposition~\ref{prop:carac-tt-negl} relativized to~$A$, $\D$ contains no $A$-Kurtz random sequence. But since $A$ is low for Kurtz randomness, $\D$ contains no Kurtz-random sequence as well.
\end{proof}

We now obtain a partial analogue of Theorem \ref{thm:k-trivial-deep-cupping}.

\begin{corollary}\label{cor:kurtz-cupping}
Let $X$ be Kurtz random and $A\in\Low(\KR)$.  Then $X \oplus A$ does not $\tt$-compute any member of any deep $\Pi^0_1$ class. 
\end{corollary}

\begin{proof}
Let $\Phi$ be a $\tt$-functional.  Since $\Phi$ is total, it is certainly total on all sequences of the form $X\oplus A$ for $X\in\cs$.  Thus $\Psi^A(X)=\Phi(X\oplus A)$ is a $\tt(A)$-functional.  By Theorem \ref{thm:kurtz-cupping}, it follows that $\Phi(X\oplus A)$ cannot be contained in any deep class.
\end{proof}

\begin{question}
Does Corollary \ref{cor:kurtz-cupping} still hold if we replace ``deep" with ``$\tt$-deep"?
\end{question}

We can extend Theorem \ref{thm:kurtz-cupping} to the following result, which proceeds by almost the same proof, the details of which are left to the reader.

\begin{theorem}
Let $X$ be Martin-L\"of random and $A$  be in $\Low(\MLR,\KR)$. Then $X$ does not $\tt(A)$-compute any member of any deep $\Pi^0_1$ class. (In particular, $X \oplus A$ does not $\tt$-compute any member of any deep $\Pi^0_1$ class). 
\end{theorem}


\section{Depth, mutual information, and the Independence Postulate}\label{sec:IP}

In this final section, we introduce the notion of mutual information and apply it to the notion of depth.  Roughly, what we prove is that every member of every deep class has infinite mutual information with Chaitin's $\Omega$, a Martin-L\"of random sequence that encodes the halting problem.  This generalizes a result of Levin's, that every consistent completion of $\PA$ has infinite mutual information with~$\Omega$.  We conclude with a discussion of the Independence Postulate, a principle introduced by Levin to derive the statement that no consistent completion of arithmetic is physically obtainable.

\subsection{The definition of mutual information}

First we review the definitions of Kolmogorov complexity of a pair and the universal conditional discrete semi-measure $\m(\cdot\mid\cdot)$. Let $\la\cdot,\cdot\ra:\str\times\str\rightarrow\str$ be a computable bijection.  Then we define $\K(\sigma,\tau):=\K(\la\sigma,\tau\ra)$.  Similarly, we set $\m(\sigma,\tau):=\m(\la\sigma,\tau\ra)$.  A conditional left-c.e.\ discrete semi-measure $m(\cdot\mid\cdot):\str\times\str\rightarrow[0,1]$ is a function satisfying $\sum_{\sigma\in\str} m(\sigma\mid\tau)\leq 1$ for every $\tau$.  Then $\m(\cdot\mid\cdot)$ is defined to be a universal conditional left-c.e.\ discrete semi-measure, so that for every conditional left-c.e.\ discrete semi-measure, there is some $c$ such that $m(\sigma\mid\tau)\leq c\cdot\m(\sigma\mid\tau)$ for every $\sigma$ and $\tau$.  Lastly, we define the conditional prefix-free Kolmogorov complexity $\K(\sigma\mid\tau)$ to be
\[
\K(\sigma\mid\tau)=\min\{|\xi|:U(\la\xi,\tau\ra)=\sigma\},
\]
where $U$ is a universal prefix-free machine.

The \emph{mutual information} of two strings~$\sigma$ and $\tau$, denoted by $\I(\sigma:\tau)$, is defined by
\[
\I(\sigma\,:\,\tau) = \K(\sigma) + \K(\tau) - \K(\sigma,\tau)
\]
or equivalently by
\[
2^{\I(\sigma\,:\,\tau)} =^{\times} \frac{\m(\sigma,\tau)}{\m(\sigma) \cdot \m(\tau)}.
\]
By the symmetry of information (see G\'acs~\cite{Gacs-notes}), we also have
\begin{equation}\label{eqn-syminf}
2^{\I(\sigma\,:\,\tau)} =^\times  \frac{\m(\sigma \mid \tau,\K(\tau))}{\m(\sigma)} =^\times  \frac{\m(\tau \mid \sigma,\K(\sigma))}{\m(\tau)}.
\end{equation}
Levin \cite{Levin2013} extends mutual information to infinite sequences by setting
\[
2^{\I(X\,:\,Y)} = \sum_{\sigma, \tau \in \str} \m^X(\sigma) \cdot \m^Y(\tau) \cdot 2^{\I(\sigma\,:\,\tau)}.
\]

\subsection{Mutual information and depth}

Recall that Chaitin's $\Omega$ can be obtained as the probability that a universal prefix-free machine will halt on a given input, that is, $\Omega=\sum_{U(\sigma){\downarrow}}2^{-|\sigma|}$, where $U$ is a fixed universal prefix-free machine.  Generalizing a result of Levin's from \cite{Levin2013}, we have:

\begin{theorem}\label{thm:mutual-info}
Let $\C$ be a $\Pi^0_1$ class and $T$ its associated co-c.e.\ tree. Suppose~$\C$ is deep, witnessed by a computable order function~$f$ such that $\m(T_{f(n)}) \leq 2^{-n}$. Then for every $Y \in \C$ and all~$n$,
\[
\I\bigl(\Omega \uh n\,:\,Y \uh f(n)\bigr) \geq n-O(\log n).
\]
In particular,
\[
\I(\Omega\,:\,Y)=\infty.
\]
\end{theorem}

\begin{proof}
Our proof follows the same idea Levin uses for consistent completions of $\PA$ (see ~\cite[Theorem 1]{Levin2013}), although some extra care is needed for arbitrary deep classes. Suppose for a given~$n$ we have an exact description~$\tau$ of $T_{f(n)}$; that is, on input $\tau$, the universal machine outputs a code for the finite set $T_{f(n)}$. By the definition of $f$,
\[
\sum_{\sigma \in T_{f(n)}} \m(\sigma) \leq 2^{-n}
\] 
or equivalently 
\[
\sum_{\sigma \in T_{f(n)}} \m(\sigma) \cdot 2^{n} \leq 1
\]
Therefore, the quantity $\m(\sigma) \cdot 2^{n} \cdot \mathbf{1}_{\sigma \in T_{f(n)}}$ is a discrete semi-measure, but it is not necessarily left-c.e.\ since $T_{f(n)}$ is merely co-c.e.\ (and, in general, not c.e. by Proposition \ref{prop:no-deep-comp-tree}). However, it \emph{is} a left-c.e.\ semi-measure relative to the exact description~$\tau$ of $T_{f(n)}$. Thus, for every $\sigma \in T_{f(n)}$, by the universality of $\m(\cdot\mid\tau)$,
\[
\m(\sigma \mid \tau) \geq^\times \m(\sigma) \cdot 2^n.
\]
By Equation \ref{eqn-syminf} above, we have
\[
2^{\I(\sigma\,:\,\tau)} =^\times  \frac{\m(\sigma \mid \tau,\K(\tau))}{\m(\sigma)}  \geq^\times  \frac{\m(\sigma \mid \tau)}{\m(\sigma)} \geq^\times2^n,
\]
and hence $\I(\tau\,:\,\sigma) \geq^+ n$.  We would like to apply this fact to the case where $\sigma = Y \uh f(n)$ and $\tau = \Omega \uh n$. But this is not technically sufficient, as $\Omega \uh n$ does not necessarily contain enough information to exactly describe $T_{f(n)}$. This is not an obstacle in Levin's argument for completions of $\PA$, but it is for arbitrary deep classes. 

However, $\Omega \uh n$ contains enough information to get a ``good enough" approximation of $T_{f(n)}$. Let us refine the idea above: suppose now that $\tau$ is no longer an exact description of $T_{f(n)}$, but is a description of a set of strings~$S$ of length~$f(n)$ such that $T_{f(n)} \subseteq S$ and 
\[
\sum_{\sigma \in S} \m(\sigma) \leq^\times 2^{-n} n^2.
\]
Then, by following the same reasoning as above, we would have $\I(\sigma : \tau) \geq^+ n-2\log n$ for all $\sigma \in S$ (and thus all $\sigma \in T_{f(n)}$). We shall prove that $\Omega \uh n$ contains enough information to recover such a set~$S$, thus proving the theorem. 

The real number $\Omega$ is left-c.e.\ and Solovay complete (see~\cite[Section 9.1]{DowneyH2010}). As a consequence, for every other left-c.e.\ real $\alpha$, knowing the first~$k$ bits of $\Omega$ allows us to compute the first $k-O(1)$ bits of $\alpha$. For all~$n$, define:
\[
a_n = \sum_{\substack{|\sigma|=f(n) \\ \sigma \notin T_{f(n)}}} \m(\sigma)
\]
and observe that $a_n$ is left-c.e.\ uniformly in~$n$ (because $T_{f(n)}$ is co-c.e.\ uniformly in~$n$), and belongs to $[0,1]$. Define now
\[
\alpha = \sum_{n\in\N} \frac{a_n}{n^2},
\]
which is also a left-c.e.\ real. Thus, knowing the first $n$ bits of $\Omega$ gives us the first $n-O(1)$ bits of $\alpha$, i.e., an approximation of $\alpha$ with precision $2^{-n}\cdot O(1)$. In particular, one can find a stage~$t_n$ such that $\alpha-\alpha[t_n] \leq 2^{-n}\cdot O(1)$, and thus $a_n-a_n[t_n] \leq 2^{-n} \cdot n^2 \cdot O(1)$. By definition of $a_n$,
\[
a_n[t_n] = \sum_{\substack{|\sigma|=f(n) \\ \sigma \notin T_{f(n)}[t_n]}} \m(\sigma)[t_n],
\]
and since $|a_n - a_n[t_n]| \leq 2^{-n} \cdot n^2 \cdot O(1)$ , this implies
\[
\sum_{\substack{|\sigma|=f(n) \\ {\sigma \in T_{f(n)}[t_n]} \setminus T_{f(n)}}} \m(\sigma) \leq^\times 2^{-n} \cdot n^2.
\]
But recall from above that 
\[
\sum_{\substack{|\sigma|=f(n) \\ {\sigma \in T_{f(n)}}}} \m(\sigma) \leq 2^{-n}.
\]
Combining these two facts, and taking~$S$ to be the set $T_{f(n)}[t_n]$, we have $T_{f(n)} \subseteq S$ and 
\[
\sum_{\sigma \in S} \m(\sigma) \leq^\times 2^{-n} \cdot n ^2,
\]
which establishes the first part of the theorem.  To see that the second part of the statement follows from the first, observe that
\[
\m^Y(Y \uh f(n)) =^\times \m^Y(n) \geq^\times \m(n) \geq^\times 1/n^2,
\] 
\[
\m^\Omega(\Omega \uh n) =^\times \m^\Omega(n) \geq^\times \m(n) \geq^\times 1/n^2,
\] and $2^{\I(Y \uh f(n)\,:\,\Omega \uh n)}=2^n/n^{O(1)}$ (since $\I\bigl(Y \uh f(n)\,:\,\Omega \uh n\bigr)\geq^+n-2\log n$ as established above). Then we have
\begin{eqnarray*}
2^{\I(\Omega\,:\, Y)} & = & \sum_{\sigma, \tau \in \str} \m^\Omega(\sigma) \cdot \m^Y(\tau) \cdot 2^{\I(\sigma\,:\,\tau)}\\
 &\geq & \sum_{n\in\N} \m^\Omega(Y\uh f(n)) \cdot \m^Y(\Omega\uh n) \cdot 2^{\I(Y\uh f(n)\,:\,\Omega\uh n)}\\
&\geq & \sum_{n\in\N} 2^n/n^{O(1)}=\infty.
\end{eqnarray*}
\end{proof}

\begin{remark}
The converse of this theorem does not hold, i.e., there is a $\Pi^0_1$ class that is not deep but all of whose elements have infinite mutual information with $\Omega$. This follows from Theorem~\ref{thm:depth-not-Turing} and the fact that having infinite mutual information with $\Omega$ is a property that is invariant under addition or deletion of a finite prefix. \\
\end{remark}

Levin's proof of Theorem~\ref{thm:mutual-info} restricted to the particular case of completions of $\PA$ is the mathematical part of a more general discussion, the other part of which is philosophical in nature (see \cite{Levin2013} and \cite{Levin1984}). While G\"odel's theorem asserts that no completion of $\PA$ can be \emph{computably} obtained, Levin's goal is to show that no completion of $\PA$ can be obtained by any physical means whatsoever (computationally or otherwise), thus generalizing G\"odel's theorem. Levin does not fully specify what he means by physically obtainable (the exact term he uses is ``located in the physical world"), but nonetheless he makes the following postulate, which he dubs the ``Independence Postulate": if $\sigma$ is a string which can be unambiguously defined by an~$n$-bit mathematical formula\footnote{A caveat: having an unambiguous mathematical definition does not mean that the string can be computably reconstructed from this description.  For example, ``the first $2^n$ bits of the halting problem'' (written as a mathematical formula) unambiguously defines an object but does not in any way give us access to the actual value of this object. Thus, an object with an~$n$-bit definition \emph{need not} have Kolmogorov complexity less than or equal to~$n+O(1)$.} (say, in ZFC) 
 and~$\tau$ can be located in the physical world with a $k$-bit description, then for a fixed small constant~$c$ (independent of~$\sigma$ and~$\tau$), one has $\I(\sigma\,:\,\tau) < n+k+c$.

In particular, if one admits that some infinite sequences are physically obtainable, the Independence Postulate for infinite sequences says that if $X$ and~$Y$ are two infinite sequences with~$X$ mathematically definable and~$Y$ physically obtainable, then $\I(X\,:\,Y) < \infty$. Being $\Delta^0_2$, $\Omega$ is mathematically definable, and, as Levin shows, $\I(\Omega\,:\,Y)=\infty$ for any completion~$Y$ of~$\PA$. Thus, assuming the Independence Postulate, no completion of~$\PA$ is physically obtainable. 

Our Theorem~\ref{thm:mutual-info} extends Levin's theorem and, assuming the Independence Postulate, shows that no member of a deep class (shift-complex sequences, compression functions, etc.) is physically obtainable. Of course, evaluating the validity of the  Independence Postulate would require an extended philosophical discussion that would take us well beyond the scope of this paper. 

In any case, whether or not the reader accepts the Independence Postulate, Theorem~\ref{thm:mutual-info} is interesting in its own right. In fact, it is quite surprising because it seems to contradict the ``basis for randomness theorem" (see~\cite[Theorem 8.7.2]{DowneyH2010}), which states that if $X$ is a Martin-L\"of random sequence and $\mathcal{C}$ is a $\Pi^0_1$ class, then there exists a member~$Y$ of $\mathcal{C}$ such that $X$ is random relative to~$Y$. If a sequence~$X$ is random relative to another sequence~$Y$, the intuition is that~$Y$ ``knows nothing about~$X$", and thus one could conjecture that $\I(X : Y) < \infty$. However, this cannot always be the case, since by Theorem~\ref{thm:mutual-info}, $\I(\Omega:Y) = \infty$ for \emph{all} members~$Y$ of a deep $\Pi^0_1$ class $\C$, even though $\Omega$ is random relative to some $Y\in\C$. 

This apparent paradox can be explained by taking a closer look at the definition of mutual information. Let~$\C$ be a deep $\Pi^0_1$ class, whose canonical co-c.e.\ tree~$T$ satisfies  $\m(T_{f(n)}) \leq 2^{-n}$ for some computable function~$f$. By Theorem~\ref{thm:mutual-info} and the symmetry of information, for every $Y\in\C$ we have
\begin{equation}\label{eq:omega1}
\K(\Omega \uh n) - \K\big(\Omega \uh n \mid Y \uh f(n), k_n\big) = \I\big(\Omega \uh n : Y \uh f(n)\big) -O(1)\geq n-O(\log n),
\end{equation}
where $k_n$ stands for $\K(Y \uh f(n))$. 
Take a~$Y \in \C$ such that $\Omega$ is random relative to~$Y$. It is well-known that a sequence~$Z$ is random if and only if $K(Z \uh n \mid n) \geq n-O(1)$ (see for example G\'acs~\cite{Gacs1980}). Applying this fact (relativized to~$Y$) to $\Omega$, we have
\[
\K^Y(\Omega \uh n \mid n) \geq n-O(1)
\]
and thus in particular that 
\[
\K\bigl(\Omega \uh n \mid Y \uh f(n)\bigr) \geq n-O(1).
\]
Since $\K(\Omega \uh n) \leq n+O(\log n)$, it follows that

\begin{equation}\label{eq:omega2}
\K(\Omega \uh n) - \K\bigl(\Omega \uh n \mid Y \uh f(n)\bigr) \leq O(\log n).
\end{equation}

The only difference between (\ref{eq:omega1}) and (\ref{eq:omega2}) is the term~$k_n=\K(Y \uh f(n))$. 
But it makes a big difference, as one can verify that 
\[
\K\bigl(\Omega \uh n \mid Y \uh f(n)\bigr)-\K\bigl(\Omega \uh n \mid Y \uh f(n), k_n\bigr)\geq n-O(\log n).
\]
Informally, while $Y$ ``knows nothing" about $\Omega$, the complexity of its initial segments, seen as a function, does. In particular, the change of complexity caused by $k_n$ implies that $\K\bigl(\K(Y \uh f(n)\bigr) \geq n - O(\log n)$ and thus $\K(Y \uh f(n)) \geq 2^n/n^{O(1)}$. \\

\vspace{5mm}

\noindent \textbf{Acknowledgements} We would like to thank Noam Greenberg, Rupert H\"olzl, Mushfeq Khan, Leonid Levin, Joseph Miller, Andr\'e Nies, Paul Shafer, and Antoine Taveneaux for many fruitful discussions on the subject. We would also like to thank the anonymous referees for a number of helpful suggestions.  Lastly, we are particularly grateful to Steve Simpson, who provided very detailed feedback on the first arXiv version of this paper.

\bibliographystyle{asl}
\bibliography{deep-classes}

\end{document}